\documentclass{svjour3}                     
\usepackage{xcolor}

\smartqed  
\usepackage{graphicx}
\usepackage{amsmath}
\usepackage{multicol}

\newcommand{\R}{\mathbb{R}}
\newcommand{\Rex}{\overline{\mathbb{R}}}

\usepackage{amssymb}
\usepackage{cite}
\DeclareMathOperator{\ri}{ri}
\usepackage{enumitem}

\let\epsilon\varepsilon
\newcommand{\N}{\mathbb{N}}
\DeclareMathOperator{\cl}{cl}

\DeclareMathOperator{\co}{co}
\DeclareMathOperator{\lin}{lin}
\DeclareMathOperator{\cco}{\overline{co}}
\DeclareMathOperator{\aff}{aff}

\DeclareMathOperator{\epi}{epi}

\DeclareMathOperator{\grafo}{gph}
\DeclareMathOperator{\dom}{dom}

\DeclareMathOperator{\inte}{int}
\DeclareMathOperator{\spn}{span}

\DeclareMathOperator{\sub}{\partial}
\usepackage{dsfont}

\begin{document}

\title{Qualification conditions-free characterizations of the $\varepsilon$-subdifferential of convex integral
	functions\thanks{This work is partially supported by CONICYT grants: Fondecyt projects Nº 1151003, 1190012, 1190110, and 1150909, Proyecto/Grant PIA AFB-170001, CONICYT-PCHA/doctorado Nacional/2014-21140621.}}

\titlerunning{Subdifferential of  convex integral functionals}        

\author{Rafael Correa, Abderrahim Hantoute, Pedro P\'erez-Aros   
}

%
	



\institute{Rafael Correa,	Universidad de O'Higgins and DIM-CMM of Universidad de Chile\\
		\email{rcorrea@dim.uchile.cl}            \\
	Abderrahim Hantoute,	Center for Mathematical Modeling  \\
		Universidad de Chile\\ 
		\email{ahantoute@dim.uchile.cl}  \\
	Pedro P\'erez-Aros,	Instituto de Ciencias de la Ingenier\'ia 
	\\Universidad de O'higgins \\
              \email{pedro.perez@uoh.com}           
}

\date{Received: date / Accepted: date}

\maketitle

\begin{abstract}
We provide formulae for the  $\varepsilon$-subdifferential of the integral function 
$
I_f(x):=\int_T f(t,x) d\mu(t),
$
where the integrand $f:T\times X \to \Rex$ is measurable in $(t,x)$ and convex in $x$. The state variable lies in a locally  convex space, possibly non-separable, while $T$ is given a structure of a nonnegative complete $\sigma$-finite measure space $(T,\mathcal{A},\mu)$. The resulting characterizations are given in terms of the $\epsilon$-subdifferential of  the data functions involved in the integrand, $f$, without requiring any qualification conditions. We also derive new formulas when some usual continuity-type conditions are in force. These results are new even for the finite sum of convex functions and for the finite-dimensional setting.

\end{abstract}
\section{Introduction}

Several problems in applied mathematics such as calculus of variations, optimal control theory and  stochastic programming, among others, rely on the study of integral functions and functionals given by  the following expression
\begin{equation}\label{DefinitionIntegrand}
	\displaystyle x(\cdot) \in \mathfrak{X} \to \hat{I}_f(x(\cdot)):=\int_T
	f(t,x(t)) d\mu(t),
\end{equation}
for a functional space $\mathfrak{X} $ and an integrand function $f$ that are given by the data of the underlying system. 
Problems which consider this class of functionals  represent a tremendous territory for developing variational
analysis, and indeed it is especially under this class of problems where the theory has traditionally been organized. Models  which consider integrals  with respect
to time are common in the study of dynamical systems and problems of optimal
control. Also, when the problem involves uncertainty, as in stochastic programming problems, the design of such mathematical models is represented by a probability space, so that the problem is modeled using the integral expression.  Applications to stochastic programming problems often concern the study of density distributions, which can also be presented under integration with respect to the Lebesgue measure.

In classical studies, such as in the calculus of variations, the integrand $f(t, x)$ is usually supposed to be continuous in $t$ and $x$, jointly, or even with some order of differentiability. Later, integrands with finite values and  satisfying the Carath\'eodory condition, that is, continuity in $x$ and measurability in $t$, have been considered. It can be easily noticed in all of these cases that for every measurable function $x(\cdot)$, the function $t \to f(t,x(t))$ is at least measurable and, hence, \eqref{DefinitionIntegrand} can be  well-defined using the convention adopted for the extended-real line. However, new mathematical models, especially the emergence of the modern control theory, lead to the consideration of general integrands with possibly infinite values. In this way, important kinds of constraints can most efficiently be represented. Such integrands require a distinctly new theoretical approach, where questions of measurability, meaning of the integral and the existence of measurable selections are prominent and are reflected in the  concept of normal integrands.

As is traditional in optimization and generally in variational analysis, one could replace the continuity of $f_t:=f(t,\cdot)$ with the weaker property of lower semicontinuity, but maintaining the measurability of $f$ with respect to $t$. Nevertheless, this is not enough to ensure the measurability of   $t \to f(t,x(t))$ for any measurable function $x(\cdot)$. As an example, put $T=[0,1]$ and let $\mathcal{A}$ be the $\sigma$-Algebra of Lebesgue measurable sets in $[0,1]$. If $D$ is a non-measurable set in $[0,1]$, and 
$$
f(t,x):=0 \text{ if }t=x\in D \text{, and }f(t,x):=1 \text{ otherwise},
$$ 
then the measurability and the lower semicontinuity of $f$ hold trivially. However, for $x(t)=t$ we lack the measurability of the function $t \to f(t,x(t))$. This example shows that although the lower semicontinuity assumption in $x$ is certainly right, the assumption of measurability in $t$ for each fixed $x$ is not adequate.
The way out of this impasse was found by Rockafellar \cite{MR0236689}, using the concept of normal convex integrands, which is an equivalent definition to the one presented  in Section \ref{Preliminaryresults} (see, e.g., \cite[Proposition 2D]{MR0512209}, or \cite[Proposition 14.39]{MR1491362}). Such integrands are such that $f(t, \cdot)$ is proper and lower semi-continuous for each $t$, and there exists a countable collection $U$ of measurable functions $u$ from $T$ to $\R^n$ with the following properties: \textbf{(a)} for each $u\in U$, $f(t,u(t))$ is measurable in $t$; \textbf{(b)} for each $t$, $U_t \cap \dom f_t$ is dense in $\dom f_t$, where $U_t=\{ u(t): u\in U \}$.

The notion of convex normal integrands provides the link that allows the connection of measurable multifunctions and  the subdifferential theory.  The preservation of the measurability  of multifunctions under a broad variety of operations,  including countable intersections and unions, sums, Painlev\'e-Kuratowski  limits, and so on, as well as the validity of Castaing's representation theorem (\cite{MR0223527,MR0467310}), have made this theory very popular in various problems of applied mathematics. Castaing's representation theorem is intrinsically related  to the possibility of extending the definition of the classical integration to the one of set-valued mappings, using measurable and integrable selections. In our case, we shall deal with the subdifferential mapping to get similar results as in the Leibniz integral rule.

Some of the classical results on integral functions and functionals can be found in Castaing-Valadier \cite{MR0467310}, Ioffe-Levin \cite{MR0372610}, Ioffe-Tikhomirov \cite{0036-0279-23-6-R02}, Levin \cite{0036-0279-30-2-R03} and Rockafellar \cite{MR0236689,MR0310612,MR0390870}. Other recent works are Borwein-Yao \cite{MR3235316},  Ioffe \cite{MR2291564}, Lopez-Thibault \cite{MR2444461}, and Mordukhovich-Sagara \cite{BorisINTErgal}  among others. A summary of the theory of measurability and integral functionals in finite-dimensions can be found in \cite{MR0512209,MR1491362}, and in \cite{MR0467310,MR0499053,MR0486391,MR577971} for infinite-dimensional  spaces.

The aim of this research is  to give formulae for the $\epsilon$-subdifferential of the convex integral function $I_f$, given by 
\begin{equation}\label{DefinitionIntegrand2}
	\displaystyle x\in X \to I_f(x):=\int_T
	f(t,x) d\mu(t),
\end{equation}
that is, when the space  $\mathfrak{X}$ in  \eqref{DefinitionIntegrand}  is the space of constant functions. This particular case is also known as the continuous sum. A well-known formula, given under certain continuity assumptions by Ioffe-Levin \cite{MR0372610} for the finite-dimensional setting, establishes that 
\begin{equation}\label{equationIoffelevinchapter4}
	\sub I_f(x)=\int_T \sub f_t(x) d\mu(t) + N_{\dom I_f}(x), \text{ for all } x\in X,
\end{equation} 
where the set $\int_T \sub f_t(x) d\mu(t)$ is understood in the sense of Aumman's integral (see  Definition \ref{AUMMANINTEGRAL}). One can compare \eqref{equationIoffelevinchapter4} with its discrete (finite) counterpart, which states that for every two proper convex functions $f_1$ and $f_2 $ such that $f_1$ is continuous at some point of $\dom f_1 \cap \dom f_2$ one has  
$$
\sub (f_1 + f_2)(x)=\sub f_1(x) + \sub f_2(x).
$$
However, in the absence of continuity assumptions, Hiriart-Urruty and Phelps formula ensures that  (see, e.g., \cite{MR1330645})
$$
\sub (f_1 + f_2)(x)=\bigcap_{\varepsilon> 0}  \cl(\sub_{\varepsilon} f_1(x) + \sub_{\varepsilon} f_2(x)),
$$
provided that $f_1$ and $f_2$ are lower semi-continuous proper convex functions.
At this step, a natural question is to give similar formulae for the subdifferential of integral functions as the last one above. It feels natural to think about a generalization of \eqref{equationIoffelevinchapter4} in the following form,
\begin{equation}\label{equationIoffelevin2chapter4}
	\sub I_f(x)=\bigcap\limits_{\eta > 0} \cl \bigg\{ \int_T \sub_\eta f_t(x) d\mu(t) + N_{\dom I_f}(x)\bigg\},
\end{equation} 
but unfortunately such an expression does not hold in general. Formula (\ref{equationIoffelevin2chapter4}) may fail even for smooth integrands, where $\sub I_f(x)$ is nonempty and at the same time $\int_T \sub_{\eta} f_t(x) d\mu(t) $ is empty (see Example \ref{ejemploimportante}).

Our main goal in this paper is to provide general formulae for the (exact) subdifferential and the $\epsilon$-subdifferential of the convex integral functional  $I_f$, defined in an arbitrary locally convex space, in the absence of qualification conditions. Other characterizations given under appropriate continuity conditions are investigated in \cite{INTECONV2}, while the nonconvex case is treated in \cite{INTENONCONV}. 

The rest of the paper is organized as follows: In  Section \ref{Notation}, we summarize the notation which is needed in the sequel. In Section \ref{Preliminaryresults}, we give some definitions and preliminary results  of vector integration, measurable multifunctions, measurable selections and integral of multifunctions, which are used to study  the subdifferential of the integral functional  $I_f$.  In Section \ref{formulaeconvexfunctionals}, we present our main formulae, which characterize the $\varepsilon$-subdifferential of the integral functional without any qualification conditions (see Theorem \ref{teorema2}). In this result, we explore the finite-dimensional reduction approach (see, e.g., \cite{MR3561780,MR2448918,perezsub,perezsub2} and the references therein).  Later on, we provide  corollaries and simplifications of our main formulae under some qualification conditions of the data (see Corollary \ref{corollaryexactsum}). General formulae for the discrete sum are derived from Theorem \ref{teorema2}. Finally, in   Section \ref{SUSLINSPACES}, we use calculus rules for the $\varepsilon$-subdifferential to get tighter formulae in the setting of Suslin spaces. We also consider there the case of countably discrete measure space to study the subdifferential of series of  convex functions (see, e.g., \cite{MR3571567}).

\section{Notation}

\label{Notation}

In this section, we give the main notation and definitions that will be used
in the sequel (see, e.g., \cite{moreau1967fonctionnelles,MR1451876,MR1921556,MR0467080,MR2184742,MR2596822,MR1865628}). 

We denote by $(X,\tau _{X})$ a Hausdorff (separated) locally convex space (lcs, for short), endowed with a given (initial) topology $\tau _{X}$, and by $X^{\ast }$ a locally convex topological dual of $X$ with respect to a given bilinear form $\langle \cdot ,\cdot \rangle :X^{\ast
}\times X\rightarrow \mathbb{R}$, defined as 
$$\langle x^{\ast },x\rangle :=\langle
x,x^{\ast }\rangle :=x^{\ast }(x).
$$
Examples of the locally convex topology $\tau _{X^{\ast }}$ that will be considered on $X^*$ are the
weak$^{\ast }\ $topology $w(X^{\ast },X)$ ($w^{\ast },$ for short) and  the Mackey topology 
$\tau (X^{\ast },X)$.  These two topologies delimit the set of compatible topologies for the pair $(X,X^*)$. Another  topology on $X^*$  (possibly not compatible for the cited pair) that we will use in the sequel is the strong topology $\beta (X^{\ast
},X).$ Bounded sets in $X^*$ with respect to the $w^{\ast }$ or $\beta (X^{\ast
},X)$ topologies are the same. The convergence with respect to the weak$^{\ast }\ $topology is denoted by $\rightharpoonup $. 
We omit the reference to the topology $\tau _{X}$, and simply write $X$ when no confusion occurs.

For a point $x\in X$, $\mathcal{N}_{x}(\tau _{X})$ represents the family of convex and symmetric
neighborhoods of $x$ with respect to the topology 
$\tau _{X}$. We omit the reference to the
used topology and write  $\mathcal{N}_{x}$ when there is no confusion. 
We will write $\overline{\mathbb{R}}:=\mathbb{R}\cup \{-\infty
,\infty \}$ and adopt the following conventions, 
$$
0\cdot \infty =0=0\cdot (-\infty
) \text{ and }\infty +(-\infty )=(-\infty )+\infty =\infty .
$$ 
For a seminorm $\rho :X\rightarrow \mathbb{R}$, $x\in X,$ and $r>0$, we denote 
$$
B_{\rho
}(x,r):=\{y\in X:\rho (x-z)\leq r\}.
$$ 
For a given function $f:X\rightarrow \overline{\mathbb{R}}$, the (effective) 
\emph{domain} of $f$ is 
$$\dom f:=\{x\in X\mid f(x)<+\infty \}.
$$ We
say that $f$ is \emph{proper} if $\dom f\neq \emptyset $ and $%
f>-\infty $, and \emph{inf-compact} if for every $\lambda \in \mathbb{R}$
the sublevel set $[f\leq \lambda ]:=\{x\in X\mid f(x)\leq \lambda \}$ is
compact. We denote by $\Gamma _{0}(X)$ the class of proper lower
semicontinuous (lsc) convex functions on $X$. The \emph{conjugate} of $f$ is
the function $f^{\ast }:X^{\ast }\rightarrow \overline{\mathbb{R}}$ defined
by 
\begin{equation*}
	f^{\ast }(x^{\ast }):=\sup_{x\in X}\{\langle x^{\ast },x\rangle -f(x)\},
\end{equation*}%
and the \emph{biconjugate} of $f$ is the function $f^{\ast \ast }:X\rightarrow \overline{\mathbb{R}},$ defined as the restriction of $(f^{\ast })^{\ast
}$ to $X$. For $\varepsilon \geq 0$, the $%
\varepsilon $-\emph{subdifferential} of $f$ at a point $x\in X$ where it is
finite is the set 
\begin{equation*}
	\partial _{\varepsilon }f(x):=\{x^{\ast }\in X^{\ast }\mid \langle x^{\ast
	},y-x\rangle \leq f(y)-f(x)+\varepsilon ,\ \forall y\in X\};
\end{equation*}%
if $f(x)$ is not finite, we set $\partial _{\varepsilon }f(x):=\emptyset $.
For a closed linear subspace $F$ of $X,$ $\cco_{F}f:F\rightarrow \overline{\mathbb{R}}$ is the closed convex hull of $f_{|_F}$, the restriction of $f$ to  $F$. 

The \emph{indicator} and the \emph{support} functions of a set $A\subseteq X$ are, respectively, 
\begin{equation*}
	\delta _{A}(x):=%
	\begin{cases}
		0\qquad & x\in A \\ 
		+\infty & x\notin A,%
	\end{cases}%
	\qquad \qquad \sigma _{A}:=\delta _{A}^{\ast }.
\end{equation*}%
The \emph{inf-convolution} of $f,g:X\rightarrow \overline{\mathbb{R}}$ is
the function 
$$f\square g:=\inf\limits_{z\in X}\{f(z)+g(\cdot -z)\};$$ it is
said to be exact at $x$ if there exists $z$ such that $f\square
g(x)=f(z)+g(x-z)$.

For a set $A\subseteq X$, we denote by $\inte(A)$, $\overline{A}$ (or $\cl A $), $\co(A)$, 
and $\cco(A)$, the \textit{interior}, the \textit{%
	closure}, the \emph{convex hull}, and the \emph{closed convex hull} of the set $A$, respectively. The %
\emph{linear} and the \emph{affine hulls} of $A$ are respectively defined by  
$$
\lin(A):=\spn\{A\}:=\cap\{F\subset X: A\subset F, \ F\text{ is linear}\},
$$
$$
\aff(A):=\cap\{F\subset X: A\subset F, \ F\text{ is affine}\}.
$$
The \emph{relative interior} of $A$, denoted by $\ri(A)$, is the interior of  $A$ with respect to the trace (or induced) topology on $\aff(A)$, if $\aff(A)$ is closed, and the empty set otherwise. Hence, when $X\equiv \mathbb{R}^n$, $\ri(A)$ is the usual relative interior (see \cite{MR1451876}). 
The \emph{polar} of $A$ is
the set 
\begin{equation*}
	A^{o}:=\{x^{\ast }\in X^{\ast }\mid \langle x^{\ast },x\rangle \leq
	1,\forall x\in A\},
\end{equation*}%
and the \emph{recession cone} of $A$, when $A$ is  convex and closed, is the set 
\begin{equation*}
	A_{\infty }:=\{u\in X\mid x+\lambda u\in A
	\text{ for all }\lambda \geq 0\},
\end{equation*}%
where $x$ is any fixed point in $A$. Finally, the $\varepsilon $-normal set of $A$ at $x$ is $\emph{N}_{A}^{\varepsilon
}(x):=\partial _{\varepsilon }\delta _{A}(x)$.

\section{Preliminary results}\label{Preliminaryresults}

In what follows, the lcs space $(X,\tau _{X})$ and its topological dual $X^{\ast }$ are defined as in Section \ref{Notation}. 

A Hausdorff topological space $S$ is said to be a Suslin space if there
exist a Polish (i.e., separable completely metrizable) space $P$ and a
continuous surjection from $P$ to $S$ (see \cite%
{MR0027138,MR0467310,MR0426084}); hence, $S$ does not need to be metrizable. For example, if $X$ is a separable Banach
space, then $(X,\left\Vert \cdot \right\Vert )$ and $(X^{\ast },w^{\ast })$
are Suslin spaces; indeed, the unit  ball in $X^{\ast }$ is weak*-separable and weak*-metrizable in that case.

Let $(T,\Sigma ,\mu )$ be a nonnegative complete $\sigma $-finite measure space, and let $L^{1}(T,\mathbb{R})$ denote the usual $L^{1}(T)$ space of equivalence classes of (scalar) Lebesgue integrable functions. Given a
function $f:T\rightarrow \overline{\mathbb{R}}$, we denote 
\begin{equation*}
	\mathcal{D}_{f}:=\{g\in L^{1}(T,\mathbb{R}):f(t)\leq g(t)\;\mu \text{-almost
		everywhere}\}.
\end{equation*}%
\begin{definition}\label{D1}
We define the upper integral of $f$ by 
\begin{equation}
	\int_{T}^{*}f(t)d\mu (t):=\inf_{g\in \mathcal{D}_{f}}\int_{T}g(t)%
	d\mu (t),  \label{defint}
\end{equation}%
whenever $\mathcal{D}_{f}\neq \emptyset $, and $\int_{T}^{*}f(t)d\mu (t):=+\infty $,  when $\mathcal{D}_{f}=\emptyset $. 
\end{definition}
Observe that when $f\in  L^{1}(T,\mathbb{R})$, the upper integral coincides with the usual Lebesgue integral. Conversely, if $f:T\rightarrow \overline{\mathbb{R}}$ is a measurable (in the usual sense) function with finite upper integrable, then, clearly, $f\in  L^{1}(T,\mathbb{R})$ (using the convention that $\inf_\emptyset=+\infty$). In other words, for every measurable function $f:T\rightarrow \overline{\mathbb{R}}$ we have that 
$$
\int_{T}^{*}f(t)d\mu (t)=\int_{T}f(t)d\mu (t).
$$
Definition \ref{D1} will be used next to introduce the integrability of appropriately measurable (see Definition  \ref{DD} below) vector-valued functions $f:T\rightarrow X$, via the upper integrability of the scalar-valued functions $\sigma_{B}(f(\cdot)) $, for weak* (or, equivalently, Mackey) bounded balanced subsets $B\subseteq X^{\ast}$. The issue is that these functions are not necessarily measurable, and so we cannot define rigorously their integral. Thus, the upper integral will be  involved to overcome this difficulty.  

\,

\,

\begin{definition}\label{DD}

(i) A function $f:T\rightarrow U,$ with $U$ being a topological space, is called
simple, if there are $k\in \mathbb{N}$, a partition $T_{i}\in \Sigma $ and
elements $x_{i}\in U$, $i=0,...,k,$ such that $
	f(t)= x_i \text{ for } t\in T_i.
	$
	
	\,
	
	\,
	
	\,
	
	(ii)
A function $f:T\rightarrow U,$ with $U$ being a topological space, is called strongly measurable (measurable, for short), if there exists a
sequence $(f_{n})_n$ of simple functions such that $%
f(t)=\lim\limits_{n \rightarrow \infty }f_{n}(t)$ for almost every (ae, for
short) $t\in T$.

	\,
	
	\,
	
	\,
	
(iii) A strongly measurable function $f:T\rightarrow X$ is said to be strongly 
integrable (integrable, for short), and we write $f\in \mathcal{L}^{1}(T,X)$, if $\int_{T}^{*}\sigma_{B}(f(t)) d\mu (t)$ is finite for every $(\beta(X^*,X))$-bounded balanced subset $B\subseteq X^{\ast}$.

	\,
	
	\,
	
	\,
	
(iv) A function $f:T\rightarrow X$ is called 
weakly or scalarly measurable (integrable, resp.), if for every $x^{\ast }\in X^{\ast },$ $%
t\rightarrow \langle x^{\ast },f(t)\rangle $ is 
measurable (integrable, resp.). 

A function $f:T\rightarrow X^*$ is called $w^*$-measurable ($w^*$-integrable, resp.), if for every $x\in X,$ the mapping $t\rightarrow \langle x,f(t)\rangle $ is 
measurable (integrable, resp.). 
	
	\,
	
	\,
	
	\,
	
(v) We define $\mathcal{L}_{w}^{1}(T,X)$ as the subspace of weakly measurable functions $f:T\rightarrow X$ such that $\int_{T}^{*}\sigma_{B}(f(t)) d\mu (t)$ is finite, for every bounded balanced subset $B\subseteq X^{\ast}$.

 We define $\mathcal{L}_{w^*}^{1}(T,X^*)$ as the subspace of w%
$^*$-integrable functions $f$ such that $\int_{T}^{*}\sigma_{B}(f(t)) d\mu (t)$ is finite, for every bounded balanced subset $B\subseteq X$. 

	\,
	
	\,
	
	\,
	
(vi) We define $L^{1}(T,X)$, $L_{w}^{1}(T,X)$ and $L_{w^*}^{1}(T,X^*)$ as the quotient spaces of $\mathcal{L}%
^{1}(T,X)$, $\mathcal{L}_{w}^{1}(T,X)$ and $\mathcal{L}_{w^*}^{1}(T,X^*)$, respectively, with respect to the equivalence relations $f=g$ ae, $\langle f,x^{\ast
}\rangle =\langle g,x^{\ast }\rangle $ ae $t$ for all $x^{\ast }\in X^{\ast },$
and $\langle f,x\rangle =\langle g,x\rangle $ ae $t$ for all $x\in X,$ respectively.
	
	\,
	
	\,
	
	\,
	
(vii) Given w$^*$%
-integrable function $f:T\rightarrow X^{\ast }$ and $E\in \Sigma $, the weak integral of $f$ over $E$
is the linear mapping $\int_{E}f(t)d\mu(t):=x_{E}^{\sharp }$, defined on $X$ as 
$$
x_{E}^{\sharp
}(x):=\int_{E}\langle f(t),x\rangle d\mu(t). 
$$ 
 When $f$ is strongly integrable; hence, w$^*$-integrable,
the element $\int_{E}fd\mu$ is also referred to as the strong integral of $f$ over $E$. 
%
%

	\,
	
	\,
	
	\,
	
(viii) When $X$ is Banach, we define $L^{\infty }(T,X)$ as the normed space (of classes of equivalence with respect to the relation $f=g$ ae) of strongly measurable functions 
$f:T\rightarrow X$, which are essentially bounded; that is, 
$$
\Vert f\Vert
_{\infty }:=\text{ess}\sup \{\Vert f(t)\Vert :t\in T\}<\infty .
$$	

(ix) When $X$ is Banach, we call an element $\lambda^{\ast }\in L^{\infty }(T,X)^{\ast }$ (the topological dual space of $L^{\infty }(T,X)$) a singular measure, and we write $\lambda^{\ast }\in L^{\text{sing}}(T,X)$, when there exists a sequence of measurable sets $T_{n}$ such that 
$$
T_{n+1}\subseteq T_{n}, \quad \mu (T_{n})\rightarrow_n 0, \quad \lambda^{\ast }(g\mathds{1}_{T^c_{n}})=0, \text{ for all }g\in L^{\infty}(T,X), \text{ and all }n\in\N,
$$ 
where $\mathds{1}_{A}$ denotes the characteristic function of $A$, equal to $1$ in $A$ and $0$ outside. 
\end{definition}

In the following remark, we gather some useful comments in order to explain the concepts introduced in the above definition. 
\begin{remark}
(i) If $X$ is a Suslin Banach space  and $X^*$ is its topological dual, then
every $(\Sigma ,\mathcal{B}(X))$-measurable function $f:T\rightarrow X$
(that is, $f^{-1}(B)\in \Sigma $ for all $B\in \mathcal{B}(X)$) is weakly
measurable. Here, $\mathcal{B}(X)$ is the Borel $\sigma $-Algebra of the
open (equivalently, weakly open) sets of $X$ (see, e.g., \cite[Theorem III.36 ]%
{MR0467310}).

(ii) In the Banach spaces setting, where $X$ is a Banach space and $X^*$ is its topological dual, $L^{1}(T,X)$ coincides with the space of Bochner integrable functions (see, e.g., \cite[\S II]{MR0453964}). If, in addition,  $X$ is separable, then both notions of (strong and weak) measurability coincide (see \cite[\S II, Theorem 2]{MR0453964}), and, consequently, $L^{1}(T,X)=L_{w}^{1}(T,X)$ in this case. Similarly, if $(X^{\ast },\Vert \cdot \Vert )$ is separable, then 
$L^{1}(T,X^{\ast })=L_{w^{\ast }}^{1}(T,X^{\ast })$. However, when the Banach space $X$ is separable, but the dual space $X^*$ is not (with respect to the dual norm), then strong and weak measurability may not coincide (see \cite[\S II Example 6]{MR0453964}), and so  $L^{1}(T,X)\neq L_{w}^{1}(T,X)$ in general. 

\,
	
	\,
	
	\,
(iii) It is clear that every strongly integrable function is weakly integrable.
However, the weak measurability of a function $f$ does not necessarily imply
the measurability of the function $\sigma _{B}(f(\cdot ))$, and so the
corresponding integral of this last function is to be understood as the upper integral  in the sense
of Definition \ref{D1}. 
	\,
	
	\,
	
	\,
(iv) In general, the weak integral $\int_{E}fd\mu$ of a w$^*$-integrable function $f:T\rightarrow X^{\ast }$ on $E\in \Sigma $, may not be in $X^*$. But, if $X$ is Banach, then $\int_{E}fd\mu\in X^*$, and is referred to in this case as the Gelfand integral of $f $ over $E$ (see \cite[\S II, Lemma 3.1]{MR0453964} and the details therein).

	\,
	
	\,

	\,
(v) When $X$ is Banach, we know that each functional $\lambda ^{\ast }\in
L^{\infty }(T,X)^{\ast }$ can be uniquely written as
$$
\lambda ^{\ast
}(x)=\int_{T}\langle \lambda _{1}^{\ast }(t),x(t) \rangle d\mu
(t)+\lambda _{2}^{\ast }(x),\;\; \text{ for all } x\in L^{\infty }(T,X),
$$
where $\lambda _{1}^{\ast }\in L_{w^{\ast
}}^{1}(T,X^{\ast })$, $\lambda _{2}^{\ast }\in L^{\text{sing}}(T,X)$, and the integral defines  the duality product between $ L_{w^{\ast}}^{1}(T,X^{\ast })$ and $L^{\infty }(T,X)$ (see, for example, \cite{MR0467310,0036-0279-30-2-R03}).

\end{remark}

\,

We shall need the following decomposability concepts \cite[Definition 3, \S VII]{MR0467310}. Recall that $\mathcal{P}(T)$ is the power set of $T$, and $X^T$ is the set of functions from $T$ to $X$.
\begin{definition}
	\label{defintion:decomposabilityC} $(i) $ Assume that $(T,\Sigma)=(\mathbb{N}
	,\mathcal{P}{(\mathbb{N}}))$. A vector space $L\subset X^T$ is said to be
	decomposable if 
	\begin{equation*}
		c_{00}(X):=\{ (x_n): \exists k_0\in \mathbb{N} \text{ such that } x_k=0,\,
		\forall k\geq k_0\}\subset L.
	\end{equation*}
	$(ii) $ Assume that $(T,\Sigma)\ne (\mathbb{N},\mathcal{P}{(\mathbb{N}}))$.
	A vector space $L$ of weakly integrable functions in $X^T$ is said to be
	decomposable if for every $u\in L$, every weakly integrable function $f\in
	X^T$ such that $f(T)$ is relatively compact, and every  set $%
	A\in \Sigma$ with finite measure, we have that $$f\mathds{1}_{A} + u\mathds{1}_{A^c}\in L.$$
\end{definition}
\begin{remark}
The specification of the decomposability above with respect to the underlying $\sigma $%
-algebra $(T,\Sigma )$ makes sense, since the two definitions may not coincide.
For instance, if $X=\mathbb{R}$ and $\mu $ is a finite measure over $(%
\mathbb{N},\mathcal{P}{(\mathbb{N}}))$, then the space $L=c_{00}(X)$ is
obviously decomposable in the sense of Definition \ref{defintion:decomposabilityC}(i), but not with respect to  Definition \ref{defintion:decomposabilityC}(ii). Indeed, the decomposability of $L$ in the
sense of  Definition \ref{defintion:decomposabilityC}(ii) would imply that $%
\ell ^{\infty }\subseteq L$.
\end{remark}
\,

\,

\,

The above measurability and integrability concepts are extended next to the framework of multifunctions. 
\begin{definition}
A multifunction $G:T\rightrightarrows X$ is called $\Sigma$-$\mathcal{B}(X,\tau_X)$%
-measurable (measurable, for simplicity) if its graph, given as 
$$
\grafo G:=\{(t,x)\in
T\times X:x\in G(t)\},
$$ 
is an element of $\Sigma \otimes \mathcal{B}(X,\tau_X)$. We
say that $G$ is weakly measurable if for every $x^{\ast }\in X^{\ast }$, $%
t\rightarrow \sigma _{G(t)}(x^{\ast })$ is a measurable function.
\end{definition}
The next definition gives the integral of multifunctions in the sense of Aumann (see, for example,\cite{MR0185073}).
\begin{definition}\label{AUMMANINTEGRAL}
	The strong and the weak integrals of a (non-necessarily measurable)
	multifunction $G:T\rightrightarrows X^{\ast }$ are given respectively by 
	\begin{align*}
		\int_{T}G(t)d\mu (t):=& \bigg\{\int_{T}m(t)d\mu (t)\in X^{\ast
		}:m\text{ is integrable}\text{ and }m(t)\in G(t)\;ae\bigg\}, \\
		(w)\text{-}\int_{T}G(t)d\mu (t):=& \bigg\{\int_{T}m(t)d\mu
		(t)\in X^{\ast }:m\text{ is w$^{\ast }$-integrable}\text{ and }m(t)\in
		G(t)\;ae\bigg\}.
	\end{align*}
\end{definition}

In what follows, we fix a function $f:T\times X\rightarrow \overline{\mathbb{R}}$. Given a vector subspace $L$ of $X^T$, we shall call an integral functional on $L$, the
extended real-valued functional $\hat{I}_{f}$, which is given as
\begin{equation}\label{ifl}
	x(\cdot )\in L\rightarrow \hat{I}_{f}(x(\cdot
	)):=\int_{T}^*f(t,x(t))d\mu (t),
\end{equation}
Hence, in particular, when the function $f(\cdot,x(\cdot))$ is measurable, we obtain 
\begin{equation}\label{ifx}
	 \hat{I}_{f}(x(\cdot
	))=\int_{T}f(t,x(t))d\mu (t).
\end{equation}
The special but interesting case when $L$ is the linear space of
constant functions will deserve a serious study. Since in such a case $\hat{I}_{f}$ can be regarded as a function on $X$, we shall denote it simply as $I_{f}$ and write
\begin{equation*}
	x\in X\rightarrow I_{f}(x):=\int_{T}^*f(t,x)d\mu (t).
\end{equation*}%
The assumptions that will be imposed on $f$ will guarantee that $\hat{I}_{f}$ and $I_{f}$ are defined by using the usual integral instead of the upper integral. 

The function $f$ is called a \emph{$\tau_X $-normal integrand} (or, simply, normal integral when no confusion occurs), if 
$f$ is $\Sigma \otimes \mathcal{B}(X,\tau_X)$-measurable and the functions $%
f(t,\cdot )$ are lsc for ae $t\in T$. In addition, if $f(t,\cdot )\in \Gamma
_{0}(X)$ for ae $t\in T$, then $f$ is called \emph{convex normal integrand}. For
simplicity, we denote $f_{t}:=f(t,\cdot )$. 
 
The following result characterizes the
Fenchel conjugate of $\hat{I}_f$ when $X$ and $X^{\ast}$ are Suslin spaces, which is known (\cite[Theorem VII-7]{MR0467310}), as well as when  
$(T,\Sigma)=(\mathbb{N},\mathcal{P}(\mathbb{N}))$. 

\begin{proposition}
Assume that either $X,X^{\ast}$ are Suslin spaces, or $%
	(T,\Sigma )=(\mathbb{N},\mathcal{P}(\mathbb{N}))$, and assume that $f$ is a convex normal integrand. Let $L\subset X^T$ and $L^{\ast}\subset {X^*}^T$ be two vector
	spaces of weakly integrable functions, such that $L$ is decomposable and the function $$t\rightarrow
	\langle v(t),u(t)\rangle $$ is integrable for every $(u,v)\in L\times
	L^*$. Let $\hat{I}_{f}$ be the integral functional defined on $L$.
	If $f:T\times X\rightarrow \overline{\mathbb{R}}$ is a
	normal integrand such that $\hat{I}_{f}(u_{0})<\infty $ for some $u_{0}\in
	L$, then the integral functional $\hat{I}_{f^{\ast}}$ defined on $L^*$ satisfies, for all $v\in
	L^{\ast}$, 
	\begin{equation*}
		\hat{I}_{f^{\ast }}(v)=\sup_{u\in L}\int_{T}\left( \langle
		u(t),v(t)\rangle -f(t,u(t)\right)) d\mu (t),
	\end{equation*}
\end{proposition}

\begin{proof}
	First, we may suppose without loss of generality (w.l.o.g.) that $\hat{I}_{f}(u_{0})\in \mathbb{R}$;
	since for otherwise, $\hat{I}_{f}(u_{0})=-\infty $ and the conclusion
	holds trivially. So, the proof in the first case of Suslin spaces follows
	from \cite[Theorem VII-7]{MR0467310}. For the proof in the second case when $%
	(T,\Sigma )=(\mathbb{N},\mathcal{P}(\mathbb{N})),$ we fix $v\in
	L^{\ast}$ and denote 
	\begin{align*}
		\delta (n):=&\langle u_{0}(n),v(n)\rangle -f(n,u_{0}(n)),\\
		\alpha :=&\sup_{u\in L}\int_{T}\left( \langle u(t),v(t)\rangle -f(t,u(t)\right)) d\mu
		(t).
	\end{align*}
	We consider the sequence $(x_{k})_k$ defined as	\begin{equation*}
		x_{k}(n)=u_{0}(n) \text{ for } n>k,\,  \text{ and } x_{k}(n):=w_{n} \text{ for } n\leq k,
	\end{equation*}
	where $w_{n}\in X$ is any vector satisfying
	$$
	\langle w_{n},v(n)\rangle
	-f(n,w_{n})\geq \max \{f^{\ast }(n,v(n))-\frac{1}{k},\delta (n)\}
	$$ 
	when $f^{\ast }(n,v(n))<+\infty$ and, for otherwise,
	$$
	\langle w_n,v(n)\rangle -f(n,w_n)\geq \max
	\{k,\delta (n)\}.
	$$ 
	Due to the decomposability of the vector space $L$ (see Definition \ref{defintion:decomposabilityC}(i)) we see that $(x_{k})\in L$.
	Then, for every $k$ and $k_{0}$ in $ \mathbb{N}$ such that $%
	k_{0}<k$ we have  
	\begin{align*}
		\alpha & \geq \int_{n\leq k_{0}}\langle x_{k}(n),v(n)\rangle
		-f(n,x_{k}(n))d\mu (n)+\int_{n>k_{0}}\left\langle
		x_{k}(n),v(n)\right\rangle -f(n,x_{k}(n))d\mu (n) \\
		& \geq \int_{n\leq k_{0}}\langle x_{k}(n),v(n)\rangle
		-f(n,x_{k}(n))d\mu (n)+\int_{n>k_{0}}\delta (n)d\mu (n).
	\end{align*}
	Hence, by taking the limit when $k\to+\infty$ we get 
	$$
	\alpha \geq \int_{n\leq
		k_{0}}f^{\ast }(n,v(n))d\mu (n)+\int_{n>k_{0}}\delta (n)d\mu (n),
		$$
	and the inequality $\alpha \geq \int_{\mathbb{N}}f^{\ast
	}(n,v(n))d\mu (n)$ follows as $k_{0}$ goes to $+\infty .$ This finishes the
	proof because the converse inequality 
	$$\alpha \leq \int_{\mathbb{N}
	}f^{\ast }(n,v(n))d\mu (n)$$ holds trivially.\qed
\end{proof}

We also recall the following result, which gives the representation of the Fenchel conjugate
of the integral functional $\hat{I}_{f}$ defined on $L^{\infty }(T,X).$ This result was first\ proved in 
\cite[Theorem 1]{MR0310612} for the case $X=\mathbb{R}^{n},$ and next in \cite[%
Theorem 4]{MR0390870} when $X$ is an  arbitrary separable reflexive Banach space. We recall that, for $s^{\ast }\in L^{\text{sing}}(T,X) (\subset (L^{\infty }(T,X))^{\ast })$,
$$
{\sigma }_{\dom \hat{I}_{f}}(s^{\ast }):= \sup\limits_{ u \in \dom \hat{I}_{f}  }  s^\ast (  u )= \sup\limits_{ u \in \dom \hat{I}_{f}  }  \langle s^\ast, u\rangle,
$$
where the last product refers to the duality product between $L^{\infty }(T,X)$ and its dual $(L^{\infty }(T,X))^{\ast }$ (we are assuming that $X$ is Banach).
\begin{theorem}
	\label{conjugateintegrandRock}Let $X$ be a separable reflexive Banach space,
	and $f:T\times X\rightarrow \mathbb{R}\cup \{+\infty \}$ be a normal convex
	integrand. Assume that the integral functional $\hat{I}_{f}$ defined on $%
	L^{\infty }(T,X)$ is finite at some point in $L^{\infty }(T,X),$ and that the integral functional $%
	\hat{I}_{f^{\ast }}$ defined on $L_{w^{\ast }}^{1}(T,X^{\ast })$ is finite at some point. Given $u^{\ast }\in (L^{\infty }(T,X))^{\ast }$, we choose $\ell ^{\ast }\in
	L_{w^{\ast }}^{1}(T,X^{\ast })$ and $s^{\ast }\in L^{\text{sing}}(T,X)$ such that  $u^{\ast }=\ell ^{\ast }+s^{\ast }$.  Then
		\begin{equation*}
		(\hat{I}_{f})^{\ast }(u^{\ast })=\int_{T}f^{\ast }(t,\ell ^{\ast
		}(t))d\mu (t)+ {\sigma }_{\dom \hat{I}_{f}}(s^{\ast }).
	\end{equation*}
\end{theorem}

A straightforward application of the above theorem gives us a representation
of the subdifferential of integrand functionals. The proof can be found (for 
$\varepsilon =0$) in \cite[Corollary 1B]{MR0310612} for the
finite-dimensional case, and in \cite[Proposition 1.4.1]{MR2444461} for
arbitrary separable reflexive Banach space. The proof of the general case $%
\varepsilon \geq 0$ is similar, and is given here for completeness.

\begin{proposition}
	\label{dualidadRock}With the assumptions of  Theorem \ref{conjugateintegrandRock}, for every\ $u\in L^{\infty }(T,X)$ and $%
	\varepsilon \geq 0,$ one has that $u^{\ast }=\ell ^{\ast }+s^{\ast }\in
	\partial _{\varepsilon }\hat{I}_{f}(u)$ (with $\ell ^{\ast }\in L_{w^{\ast
	}}^{1}(T,X^{\ast })$ and $s^{\ast }\in L^{\text{sing}}\ (T,X)$) if and only
	if there exist an integrable function $\varepsilon _{1}:T\rightarrow
	\lbrack 0,+\infty )$ and a constant $\varepsilon _{2}\geq 0$ such that 
	\begin{equation*}
		\ell ^{\ast }(t)\in \partial _{\varepsilon _{1}(t)}f(t,u(t))\text{ ae, }%
		s^{\ast }\in {N}_{{\dom}\hat{I}_{f}}^{\varepsilon _{2}}(u),%
		\text{ and }\int_{T}\varepsilon _{1}(t)d\mu (t)+\varepsilon _{2}\leq
		\varepsilon .
	\end{equation*}
\end{proposition}
\begin{proof}
	Take $u^{\ast }=\ell ^{\ast }+s^{\ast }$ in $\partial _{\varepsilon }\hat{I}%
	_{f}(u);$ hence, $u\in \dom\hat{I}_{f}.$ Then, by Theorem \ref{conjugateintegrandRock} and the definition of $\varepsilon $-subdifferentials, we have 
	\begin{equation*}
		\int_{T}\left( f(t,u(t))+f^{\ast }(t,\ell ^{\ast }(t))-\langle \ell
		^{\ast }(t),u(t)\rangle \right) d\mu (t)+\left( \mathrm{\sigma }_{\dom\hat{I}_{f}}(s^{\ast })-\langle s^{\ast },u\rangle \right) \leq
		\varepsilon .
	\end{equation*}%
	Hence, we conclude by setting\ $\varepsilon _{1}(t):=f(t,u(t))+f^{\ast
	}(t,\ell ^{\ast }(t))-\langle \ell ^{\ast }(t),u(t)\rangle $ $(\geq 0)$ and $%
	\varepsilon _{2}:=\mathrm{\sigma }_{\dom \hat{I}_{f}}(s^{\ast
	})-\langle s^{\ast },u\rangle $ $(\geq 0).$\qed
\end{proof}

The next result, also given in \cite[Theorem 2]{MR0310612}, will be used in
the proof of  Theorem \ref{teorema2} below.

\begin{theorem}
	\label{conjugateintegrandRock2} Let $f:T\times \mathbb{R}^{n}\rightarrow 
	\mathbb{R}\cup \{+\infty \}$ be a normal convex integrand. Assume that $\bar{%
		u}\in L^{\infty }(T,\mathbb{R}^{n}),$ and that for some $r>0$ the function\ $%
	f(\cdot ,\bar{u}(\cdot )+x)$ is integrable for every $x\in \mathbb{R}^{n}$ such that $\Vert
	x\Vert <r.$ Then there is some $u^{\ast }$ in $L^{1}(T,\mathbb{R}^{n})$ such
	that the integral functional $\hat{I}_{f^{\ast }}$ defined on $L^{1}(T,\mathbb{R}^{n})$ satisfies  $\hat{I}_{f^{\ast }}(u^{\ast })<\infty$. Moreover, the integral functional $\hat{I}_{f}$ defined on $L^{\infty }(T,\mathbb{R}^{n})$
	is continuous (in the $L^{\infty }(T,\mathbb{R}^{n})$-norm) at every $u\in L^{\infty }(T,\mathbb{R}^{n})$ such that $%
	\Vert u-\bar{u}\Vert _{\infty }<r.$
\end{theorem}

The next  result deals with measurable selections in  Suslin spaces.

\begin{proposition}
	\cite[Theorem III.22]{MR0467310}\label{measurableselectiontheorem} Let $S$
	be a Suslin space and $G:T\rightrightarrows S$ be a measurable multifunction
	with non-empty values. Then there exists a sequence $(g_{n})$ of $(\Sigma , 
	\mathcal{B}(S))$-measurable functions such that $%
	\{g_{n}(t)\}_{n\geq 1}$ is dense in $G(t)$ for every $t\in T.$
\end{proposition}

\begin{lemma}
	\label{lemma2}Let $(F,\Vert \cdot\Vert _{F})\subset X$ be a finite-dimensional Banach subspace, with $(F^{\ast },\Vert \cdot \Vert _{F^{\ast }})$ being its dual, and denote by $P:X\rightarrow F$ and $P^{\ast }:F^{\ast }\rightarrow X^{\ast }$ the continuous linear projection and its adjoint mapping, respectively.
	 Then there exists a neighborhood $W\subset X$ of $0$ (depending only on $P$ and $F$) such that, for every
	integrable function\ $u^{\ast }(\cdot ):T\rightarrow F^{\ast }$, the
	composite function $P^{\ast }\circ u^{\ast }(\cdot )$ is integrable and
	satisfies 
	$$\sigma _{W}(u^{\ast }(t)\circ P)\leq \Vert u^{\ast }(t)\Vert
	_{F^{\ast }}.$$
\end{lemma}

\begin{proof}
	Since the projection mapping $P:X\rightarrow F$ is continuous, there exists a neighborhood $W\in \mathcal{N}_{0}(\tau)$ such that $\Vert P(x)\Vert
	_{F}\leq \sigma _{W^{\circ }}(x)$ for all $x\in X$. Hence, 
	\begin{equation*}
		\sigma _{W}(u^{\ast }(t)\circ P)=\sup_{x\in W}\langle u^{\ast
		}(t),P(x)\rangle \leq \sup_{y\in B_{F}(0,1)}\langle u^{\ast
		}(t),y\rangle =\Vert u^{\ast }(t)\Vert _{F^{\ast }},
	\end{equation*}
	where $B_{F}(0,1)$ is the unit ball in $F$.
	We are done since the function $P^{\ast }\circ u^{\ast }(\cdot )$ inherits
	the measurability from $u^{\ast }.$\qed
\end{proof}
\begin{lemma}
	\label{medibilidadgrafo}Assume that both $X$ and $X^{\ast }$ are Suslin and
	let $u^{\ast }:T\rightarrow X^{\ast }$ be a weak*-measurable function. Set 
	\begin{align*}
		\mathfrak{G}:=\{(t,x^{\ast },y^{\ast },v^{\ast })\in T\times X^{\ast }\times
		X^{\ast }\times F^{\ast }\mid x^{\ast }+y^{\ast }+P^{\ast }(v^{\ast
		})=u^{\ast }(t)\}.
	\end{align*}
	Then $ \mathfrak{G}\in \Sigma \otimes \mathcal{B}(X^{\ast }\times X^{\ast
	}\times F^{\ast })$, and for every measurable multifunctions $C_{1},C_{2}:T
	\rightrightarrows X^{\ast }$, $C_{3}:T\rightrightarrows F^{\ast },$ the
	multifunction $C:T\rightrightarrows X^{\ast }\times X^{\ast }\times F^{\ast
	} $, defined as 
	\begin{equation*}
		(x^{\ast },y^{\ast },z^{\ast })\in C(t)\Leftrightarrow (x^{\ast },y^{\ast
		},z^{\ast })\in C_{1}(t)\times C_{2}(t)\times C_{3}(t)\ \text{and }u^{\ast
		}(t)=x^{\ast }+y^{\ast }+P^{\ast }(z^{\ast }),
	\end{equation*}
	is measurable.
\end{lemma}

\begin{proof}
	Consider the functions $g$ and $h$ defined as 
	$$g(t,x^{\ast },y^{\ast },v^{\ast })=x^{\ast }+y^{\ast
	}+P^{\ast }(v^{\ast })-u^{\ast }(t),\,\,(t,x^{\ast },y^{\ast },v^{\ast })\in T\times X^{\ast }\times
		X^{\ast }\times F^{\ast },$$ 
	$$
	h(x^{\ast },y^{\ast },v^{\ast
	})=x^{\ast }+y^{\ast }+P^{\ast }(v^{\ast }),\,\, (x^{\ast },y^{\ast },v^{\ast })\in  X^{\ast }\times
		X^{\ast }\times F^{\ast }.
	$$
	We claim that $g$ is $(\Sigma
	\otimes \mathcal{B}(X^{\ast }\times X^{\ast }\times F^{\ast }),\mathcal{B}%
	(X^{\ast}))$-measurable. First, assume that $u^*$ is a simple function; that is,
	there exists a measurable partition of $T$, $\{T_i\}_{i=1}^n$, and elements $%
	u^*_i\in X$ such that $u^*(t)=\sum u_i^* \mathds{1}_{T_i}(t)$. Then it is easy
	to see that, for every open set $U$ on $X^*$, 
	$$g^{-1}(U)=\bigcup_{i=1}^n T_i
	\times h^{-1}(U+u_i^*) \in \Sigma \otimes \mathcal{B}(X^{\ast }\times
	X^{\ast }\times F^{\ast }).$$ Therefore $g$ is $(\Sigma
	\otimes \mathcal{B}(X^{\ast }\times X^{\ast }\times F^{\ast }),\mathcal{B}%
	(X^{\ast}))$-measurable. More generally, since $u^*$ is measurable, we can write it as the limit of a sequence of
	simple functions $u^*_n$, by 
	\cite[Theorem III.36 ]{MR0467310}. So, by considering a countable dense set $D$ on $X$ and some $\epsilon_n \to 0^+$, we
 write
	\begin{equation*}
		\mathfrak{G} = \bigcap\limits_{v \in D} \bigcap\limits_{n  \geq 1 }
		\bigcup\limits_{j \in \mathbb{N} } \bigcap\limits_{ k \geq j} \big\{ (
		t,x^{\ast },y^{\ast },v^{\ast }) \mid | \langle x^{\ast }+y^{\ast }+P^{\ast }(v^{\ast })-u_k^{\ast
		}(t) , v \rangle | < \epsilon_n \big\},
	\end{equation*}
and with this we conclude, thanks to the first part of the proof.\qed 
\end{proof}

\section{Characterizations via $\protect\varepsilon$-subdifferentials\label{formulaeconvexfunctionals}}
In this section, we characterize the $\epsilon$-subdifferential ($\epsilon\ge 0$) of the convex function $I_{f}:X\to  \overline{\mathbb{R}}$, defined in (\ref{ifx}) by  
\begin{equation*}
	 I_{f}(x)=\int_{T}f(t,x)d\mu (t).
\end{equation*}
where $X,\,X^{\ast }$ are two lcs paired in duality, and\ $f:T\times X\rightarrow \overline{\mathbb{R}}$ is a convex
normal integrand with respect to $\Sigma \otimes \mathcal{B}(X,\tau _{X}).$

We start with the main result of this section, in which we use the following notation, for $x\in X$ and $ \eta\ge0$,
$$
\mathcal{F}(x):=\{L\subseteq X:L\text{ is a finite-dimensional linear
		space and }x\in L\},$$ 
		$$
		\mathcal{I}(\eta ):=\left\{\ell \in L^{1}(T,\mathbb{R}%
	_{+}):\int_T \ell(t) d\mu(t) \leq \eta \right\}.
	$$

\begin{theorem}
	\label{teorema2} For every $x\in X$ and $\varepsilon \geq 0$ we have 
	\begin{align}
		\partial _{\varepsilon }I_{f}(x)& =\bigcap\limits_{L\in \mathcal{F}%
			(x)}\bigcup\limits_{\substack{ \epsilon =\epsilon _{1}+\epsilon _{2}  \\ %
				\epsilon _{1},\epsilon _{2}\geq 0  \\ \ell \in \mathcal{I}(\epsilon _{1})}}%
		\left\{ \int_{T}\partial _{\ell (t)}(f_{t}+\delta _{\aff\{L\cap 
			\dom I_{f}\}})(x)d\mu (t)+N_{L\cap {\dom}I_{f}}^{\epsilon
			_{2}}(x)\right\}  \label{formula2.1} \\
		& =\bigcap\limits_{L\in \mathcal{F}(x)}\bigcup\limits_{\substack{  \\ \ell
				\in \mathcal{I}(\epsilon )}} \int_{T}\partial _{\ell
			(t)}(f_{t}+\delta _{L\cap {\dom}I_{f}})(x)d\mu (t).
		\label{formula2.2}
	\end{align}%
	\end{theorem}
\begin{proof} We fix  $x\in X$ and $\varepsilon \geq 0$. First, it can be easily checked that the right-hand side in  \eqref{formula2.1} is included in the right-hand side of \eqref{formula2.2}, which is in turn included in  $\partial _{\varepsilon }I_{f}(x)$. Thus, if $\partial _{\varepsilon }I_{f}(x)$ is empty, then we are done with the proof of the theorem. In other words, we only need to prove the inclusion ``$\subset$" in \eqref{formula2.1} when $\partial _{\varepsilon }I_{f}(x)$ is nonempty; hence, $x\in  {\dom}I_{f}$. 

We may suppose that $x=0$ (the origin vector in $X$). Let us first consider the case when
$$L\cap {\dom}I_{f}=\{0\}, \text{ for every }L\in \mathcal{F}(0),$$ 
which easily leads to ${\dom}I_{f}=\{0\}$. Then, since $0\in \dom f_t$ for ae $t\in T$, we obtain that 
$$
X^*=\partial _{\varepsilon }I_{f}(0)=\int_{T}\partial _{\ell (t)}(f_{t}+\delta _{\aff\{L\cap 
			\dom I_{f}\}})(0)d\mu=\int_{T}\partial _{\ell
			(t)}(f_{t}+\delta _{L\cap {\dom}I_{f}})(0)d\mu ,
$$
and (\ref{formula2.1}) and (\ref{formula2.2}) hold trivially. 

Now, we assume that $L\cap {\dom}I_{f}\ne\{0\},$ for some $L\in \mathcal{F}(0)$; that is, 
$$
\mathcal{F}^*(0):=\{L\in \mathcal{F}(0):L\cap {\dom}I_{f}\not=\{0\}\}\ne \emptyset.
$$ 
Observe that for $L\in \mathcal{F}(0)$ being such that $L\cap {\dom}I_{f}$ is a singleton, as $0\in  {\dom}I_{f}$ we have that $L\cap {\dom}I_{f}=\{0\}$, and so  
$$\partial _{\varepsilon }(I_{f}+\delta _{\spn\{L\cap {\dom}I_{f}\}})(0)=\partial _{\varepsilon }(I_{f}+\delta _{L\cap {\dom}I_{f}})(0)=X^*.
	$$ 
Consequently, we obtain that 
\begin{eqnarray*}
		\partial _{\varepsilon }I_{f}(0) &= &\bigcap\limits_{L\in \mathcal{F}(0)}\partial _{\varepsilon }(I_{f}+\delta _{ L\cap {\dom}I_{f}})(0) \\
		&=&\bigcap\limits_{L\in \mathcal{F}(0)}\partial _{\varepsilon }(I_{f}+\delta _{\spn\{L\cap {\dom}I_{f}\}})(0) \\
&=&\bigcap\limits_{L\in \mathcal{F}^*(0)}\partial _{\varepsilon }(I_{f}+\delta _{ \spn\{L\cap {\dom}I_{f}\}})(0).
	\end{eqnarray*}	
In other words, we only need to characterize  the set $\partial _{\varepsilon }(I_{f}+\delta _{ \spn\{L\cap {\dom}I_{f}\}})(0)$, for a fixed $L\in \mathcal{F}^*(0)$. For this aim we denote 
$$F:=\spn\{L\cap {\dom}I_{f}\},$$ 
$$\hat{f}(t,z):=f_{\mid T\times F}(t,z),\,\, (t,z)\in T\times F,$$
and take $x^*\in \partial _{\varepsilon }(I_{f}+\delta _F)(0)$. We also consider a continuous projection $P:X\rightarrow F$, with an adjoint $P^*:F^*\rightarrow X^*$, together with the immersion mapping $i_F:F\rightarrow X$ of $F$ in $X$. Then the restriction of $x^*$ to $F$, denoted by  $x_{|_{F}}^{\ast}$, satisfies for all $z\in F$
\begin{eqnarray*}
\langle x_{|_{F}}^{\ast}, z\rangle&=&\langle x^{\ast }, i_F(z)\rangle \\
&\le& (I_{f}+\delta _{F})(i_F(z))-(I_{f}+\delta _{F})(0) +\epsilon \\
&=&I_{f}(i_F(z))-I_{f}(0)+\epsilon \\
&=&I_{\hat{f}}(z)-I_{\hat{f}}(0)+\epsilon; 
\end{eqnarray*}
that is, 
\begin{equation}\label{eqh}
x_{|_{F}}^{\ast}\in \partial _{\epsilon }I_{\hat{f}}(0).
\end{equation}
Because ${\dom}I_{\hat{f}}=F\cap {\dom}I_{f}\ni 0$, 
	$$
	\spn\{{\dom}I_{\hat{f}}\}=\aff\{{\dom}I_{f}\cap L\}=\spn\{{\dom}I_{f}\cap L\}=F,
	$$ 
	and $F$ is a finite-dimensional subspace, we know that $I_{\hat{f}}$ is continuous on $%
	\ri({\dom}I_{\hat{f}})$; that is to say, there exist $\eta >0$ and $%
	x_{0}\in {\dom}I_{f}\cap L$ such that $x_{0}+\eta \co\{\pm
	e_{i}\}\subseteq {\dom}I_{\hat{f}}$, where $\{\pm
	e_{i}\}$ is a given basis on $F$. Hence, if $h\in F$ belongs to $%
	\eta \co\{\pm e_{i}\}$ we have that $f(\cdot ,x_{0}+h)$ is integrable. So,
	by applying  Theorem \ref{conjugateintegrandRock2}, we obtain that $\hat{I}_{\hat{f}}$
	is continuous in a neighborhood of $x_{0}$ (in $L^{\infty }(T,F)$), and the
	hypotheses of  Theorem \ref{conjugateintegrandRock} are satisfied. Then, applying the composition rule to $I_{\hat{f}}$, (\ref{eqh}) implies that 
	$$
	x_{|_{F}}^{\ast}\in \partial _{\varepsilon
	}I_{\hat{f}}(0)=A^{\ast }(\partial _{\epsilon }\hat{I}_{\hat{f}}(0)),$$
	 where 
	$A:F\rightarrow L^{\infty }(T,F)$ is given by $A(h)=h\mathds{1}_{T}$, and $A^*: L^{\infty }(T,F)^*\rightarrow F^*$ is its adjoint defined for $h\in F$ as 
	$$%
	A^{\ast }(u^{\ast }+v^{\ast })(h)=\int_{T}\langle u^{\ast }(t),h\rangle +v^{\ast }(h%
	\mathds{1}_{T}),\,\,
	u^{\ast }\in L^{1}(T,F^{\ast }),\,\,v^{\ast }\in
	L^{\text{sing}}(T,F).$$ 
	Consequently,  taking into account Proposition \ref{dualidadRock}, there are $\alpha ^{\ast }\in L^{1}(T,F^{\ast })$
	and $\beta ^{\ast }\in L^{\text{sing}}(T,F)$ at the same time as 
	$\epsilon _{1},\epsilon _{2}\geq
	0 $, with  $ \epsilon _{1}+ \epsilon _{2}= \epsilon,$
 and $\ell \in \mathcal{I}(\epsilon _{1})$  such that
\begin{equation}\label{eqhh}
 x_{|_{F}}^{\ast }(h)=\int_{T}\langle \alpha ^{\ast }(t),h\rangle d\mu
	(t)+\beta ^{\ast }(h\mathds{1}_{T}),\,\, \text{ for all }h\in F,
\end{equation}
and 
$$\alpha^\ast (t)\in
	\partial _{\ell (t)}\hat{f}_{t}(0) \text{ ae},   
\,\,	\beta^{\ast }  \in N^{\epsilon _{2}}_{%
		{\dom}\hat{I}_{\hat{f}}}(0).$$
				Let $z^{\ast }\in
	{X^{\ast }}^T$ and $\lambda ^{\ast }\in {\R}^X$ be defined such that, for ae $t\in T$ and   $u\in X$,
	$$z
	^{\ast
	}(t)=P^{\ast }(\alpha ^{\ast }(t))=\alpha ^{\ast }(t)\circ P,\,\,  \lambda
	^{\ast }(u)=\beta^* (P(u)\mathds{1}_{T}).$$	
	Then we verify that $\lambda ^{\ast }\in X^{\ast }$ and, by  Lemma \ref{lemma2}, that $z^{\ast }\in L^{1}(T,X^{\ast })$. Now, from the fact
	that $\alpha^\ast (t)\in \partial _{\ell (t)}\hat{f}_{t}(0)$, ae, and the
	definition of $z^{\ast }$, we obtain that for ae $t\in T$ and all $z\in F$
	\begin{eqnarray*}
\langle z^{\ast }(t), z\rangle&=&\langle \alpha ^{\ast }(t), P(z)\rangle \\
&\le& \hat{f}_{t}(P(z))-\hat{f}_{t}(0) +\ell (t) \\
&=&(f_{t}+\delta _{F})(z)-(f_{t}+\delta _{F})(0)+\ell (t) . 
\end{eqnarray*}
Hence, since the last inequality trivially holds for all $z\in X\setminus F$, we infer that 	$z^{\ast }(t)\in \partial _{\ell
		(t)}(f_{t}+\delta _{F})(0)$ ae. 
		Moreover, since $A(P({\dom}I_f))\subseteq 
	{\dom}\hat{I}_{\hat{f}}$, by arguing as above we also deduce that $\lambda ^{\ast }\in N^{\epsilon _{2}}_{%
		{\dom}I_{f}\cap L}(0)$. 
		Consequently, using (\ref{eqhh}) it follows that for all $h\in X$
		\begin{eqnarray*}
		(P^*(x_{|_{F}}^{\ast }))(h)=x_{|_{F}}^{\ast }(P(h))&=&\int_{T}\langle \alpha ^{\ast }(t),P(h)\rangle d\mu
	(t)+\beta ^{\ast }(P(h)\mathds{1}_{T}),\\
	&=&\int_{T}\langle P^*\alpha ^{\ast }(t),h\rangle d\mu
	(t)+\beta ^{\ast }(P(h)\mathds{1}_{T}),\\
	&=&\int_{T}\langle z^{\ast }(t),h\rangle d\mu
	(t)+\lambda ^{\ast }(h);
	\end{eqnarray*}
	showing that $$P^*(x_{|_{F}}^{\ast })=\int_T z^{\ast }(t)d\mu(t)+\lambda^*
	\subset \int_{T}\partial _{\ell (t)}(f_{t}+\delta _{F})(0)d\mu (t)+N_{{\dom}I_{f}\cap L}^{\epsilon _{2}}(0).
	$$
But we have that $x^*-P^*(x_{|_{F}}^{\ast })\in F^{\perp }$, and so
\begin{eqnarray*}
x^{\ast }&\in &
	\int_{T}\partial _{\ell (t)}(f_{t}+\delta _{F})(0)d\mu (t)+N_{{\dom}I_{f}\cap L}^{\epsilon _{2}}(0)+F^{\perp }\\
	&=& \int_{T}\partial _{\ell (t)}(f_{t}+\delta _{F})(0)d\mu (t)+N_{{\dom}I_{f}\cap L}^{\epsilon _{2}}(0).
		\end{eqnarray*}
		This completes the proof of \eqref{formula2.1}.\qed
\end{proof}

\begin{remark}
	\label{remarkrest}As it can be easily seen from the proof of  Theorem \ref{teorema2}, instead of assuming\ that $f:T\times X\rightarrow \mathbb{R\cup
		\{+\infty \}}$ is a normal convex integrand, it is sufficient to suppose
	that for every finite-dimensional subspace $F$ of $X$, the function $%
	f_{|_{F}}:T\times F\rightarrow \mathbb{R\cup \{+\infty \}}$ is a convex
	normal integrand; of course, both assumptions coincide in the finite-dimensional
	setting, but they are not equivalent in general.
\end{remark}

\begin{remark}
	\label{remaknumerabilidad} It is worth mentioning that Theorem \ref{teorema2} above also
	holds if, instead of $\mathcal{F}(x)$, we take some subfamily of
	finite-dimensional subspaces, $\tilde{\mathcal{L}}:=\{L_n,\, n\ge1 \}\subseteq \mathcal{F}(x)$, such that $$%
	\bigcup_{n\in \mathbb{N}}{L_{n}}=X.$$ For example, if the space $X$
	is separable, or more generally, if $\epi I_{f}$ is separable, then we take 
a dense set $(x_{i},\alpha _{i})_{i\geq 1}$ in $\epi I_{f}$, and define $$L_{n}:=%
	\spn\{x,x_{i},\, i=1,\cdots, n\}.$$ Then it is easy to see that $\bigcap\limits_{n\in 
		\mathbb{N}}\partial _{\epsilon }(I_{f}+\delta _{L_{n}})(x)=\partial
	_{\varepsilon }I_{f}(x)$.
\end{remark}

	In  Theorem \ref{teorema2} one can weaken the convexity hypothesis\ by
	assuming that, for every finite-dimensional subspace $F\subset X,$ $%
	f_{|_{F}} $ is a normal integrand and 
	\begin{equation}\label{eqt}
		\cco_{F}I_{f}=I_{\cco_{F}f}.
	\end{equation}%
	\begin{corollary}
Under condition (\ref{eqt}), we have that 
	\begin{align*}
		\partial _{\varepsilon }I_{f}(x)& \subseteq \bigcap\limits_{L\in \mathcal{F}%
			(x)}\bigcup\limits_{\substack{ \epsilon =\epsilon _{1}+\epsilon _{2}  \\ %
				\epsilon _{1},\epsilon _{2}\geq 0  \\ \ell \in \mathcal{I}(\epsilon _{1})}}%
		\Bigg\{ \int_{T}\partial _{\ell (t)+m_{x,L}(t)}(f_{t}+\delta _{\aff%
			\{L\cap {\dom}I_{f}\}})(x)d\mu (t)\\&\hspace{4cm}+N_{{\dom}I_{f}\cap
			L}^{\epsilon _{2}}(x)\Bigg\} , \\
		\partial _{\varepsilon }I_{f}(x)& \subseteq \bigcap\limits_{L\in \mathcal{F}%
			(x)}\bigcup\limits_{\substack{  \\ \ell \in \mathcal{I}(\epsilon )}}\left\{
		\int_{T}\partial _{\ell (t)+m_{x,L}(t)}(f_{t}+\delta _{L\cap {		\dom}I_{f}})(x)d\mu (t)\right\} ,
	\end{align*}%
	where $$m_{x,L}(\cdot ):=f(\cdot ,x)-\cco_{L}f(\cdot ,x)$$ is
	the modulus of convexity over $F.$ In addition, if $\partial
	_{\varepsilon }I_{f}(x)\neq \emptyset ,$ then $$\int_{T}m_{x,L}(t)d\mu
	(t)\leq \epsilon, \text{ for all } L\in \mathcal{F}(x).$$ 
	\end{corollary}
	\begin{proof}
	The first statement being straightforward from Theorem \ref{teorema2}, we only prove the last one. Take $L\in 
	\mathcal{F}(x)$. Then, for every $t\in T$, we have that $$f(t,x)\geq \cco%
	(f(t,\cdot )_{|_{L}})(x)\geq \cco f(t,x).$$ Also, the
	non-emptiness of $\partial _{\varepsilon }I_{f}(x)\ $ensures that $%
	I_{f}(x)\leq \cco I_{f}(x)+\epsilon ,$ so that the hypothesis $I_{\cco f}(x)=%
	\cco I_{f}(x)$ leads to $I_{f}(x)\leq I_{\cco f}(x)+\epsilon $. This\
	implies that\ $$\int_{T}\bigg( f(t,x)-\cco(f(t,\cdot )_{|_{L}})(x)\bigg)d\mu
	(t)\leq \epsilon .$$ In particular,\ if $f$ is a convex normal integrand, or
	if $\partial I_{f}(x)\neq \emptyset ,$\ then we get $m_{x,L}(\cdot )=0.$\qed
	\end{proof}

The next example, given for smooth functions $f_t$ defined on finite-dimensional spaces, justifies the use of the indicator function inside the integral symbol in the formulae of Theorem  \ref{teorema2}.

\begin{example}
	\label{ejemploimportante}Consider the function $f(x):=\frac{b}{a}x+b+\delta
	_{\lbrack -\eta ,\eta ]}(x)$, $a,b,\eta >0$. Then we have 
	\begin{equation*}
		\partial _{\varepsilon }f(0)=\left[ -\frac{\epsilon }{\eta }+\frac{b}{a},%
		\frac{\epsilon }{\eta }+\frac{b}{a}\right] .
	\end{equation*}%
	We consider the Lebesgue measure on $]0,1]$ and the convex normal integrand $f:\,]0,1]\times \mathbb{R}\rightarrow
	\lbrack 0,+\infty ]$ given by 
	$$f(t,x):=\frac{b(t)}{a(t)}x+b(t)+\delta
	_{\lbrack -\eta (t),\eta (t)]}(x),$$
	 where $a(t)=\eta (t)=t$ and $b(t)=\frac{1}{\sqrt{t}}+1.$ Hence, 
	\begin{equation*}
		I_{f}(x)=\left\{ 
		\begin{array}{crr}
			\int_{0}^{1}\left( 1+\frac{1}{\sqrt{t}}\right) dt & \text{ if }~ & x=0 \\ 
			+\infty & \text{ if }~ & x\neq 0,%
		\end{array}%
		\right.
	\end{equation*}%
	and we obtain $\partial I_{f}(0)=\mathbb{R}$, while
	\begin{equation*}
		\partial _{\varepsilon }f_t(0)=\left[ -\frac{\epsilon }{t }+\frac{\frac{1}{\sqrt{t}}+1}{t},%
		\frac{\epsilon }{t }+\frac{\frac{1}{\sqrt{t}}+1}{t}\right] =\left[ \frac{1-\epsilon}{t }+\frac{1}{t^{3/2}},
		\frac{1+\epsilon }{t }+ \frac{1}{t^{3/2}}\right].
	\end{equation*}%
	Consequently, the set\ $\int_0^1 \partial
	_{\varepsilon }f_t(0) dt$ is empty for every $0<\epsilon <1$.
\end{example}

\begin{remark}
	\label{remcon}Observe that the formulae of  Theorem \ref{teorema2} can be
	simplified if some qualification conditions (QC, for short) are in force.
	For instance, each one of conditions (\ref{QC0})-(\ref{QC3}) below (see \cite[Theorem 2.8.3]{MR1921556}) ensures the validity of the exact sum rule 
	$$\partial _{\epsilon
	}(f_{t}+\delta _{\aff ( L\cap \dom
		I_f ) })(x)=\partial _{\varepsilon }f_{t}(x)+\partial \delta _{%
		\aff{( L\cap \dom
			I_f )	}}(x), \,\,t\in T,\,\,L\in \mathcal{F}(x),$$ which in turn gives rise to characterizations of $\partial
	_{\varepsilon }I_{f}(x)$ by means only of the $\varepsilon $%
	-subdifferentials of the $f_{t}$'s:
	
	\begin{enumerate}[label={QC(\roman*)},ref={QC(\roman*)}]
		
		\item\label{QC0} $X=\mathbb{R}^{n}$ and $\ri(\dom f_{t})\cap \aff ( L\cap \dom
		I_f) \neq \emptyset $.
		
		\item\label{QC1} $X$ Banach and $\mathbb{R}%
		_{+}(\dom f_{t}-\aff(\dom I_{f}\cap L))$ is a closed
		subspace.
		
		\item\label{QC2} $f_{t}$ is
		continuous at some point of $\dom I_{f}$.
		
		\item\label{QC3} For every $B\in \mathcal{N}_{0}$, there
		exist $r >0$ and $V\in \mathcal{N}_{0}$ such that $$V\cap \spn\bigg\{%
		\dom f_{t}-\aff(\dom I_{f}\cap L)\bigg\}\subseteq
		\{f_{t}\leq r\}\cap B-\aff(\dom I_{f}\cap L).$$
	\end{enumerate}
	
	All of the above conditions imply the following property (see, e.g., \cite{MR1721727,MR2197287,MR2346533,MR3561780,MR3522006}): 
	\begin{enumerate}[label={QC(\roman*)},ref={QC(\roman*)}]
		\setcounter{enumi}{4}
		\item \label{QC4.0} For every $x^{\ast }\in X^{\ast }$,
		\begin{align}\label{QC4}
			\big(f_t+\delta_{\aff(\dom I_{f}\cap L)}\big)^{\ast }(x^{\ast })&=\min %
			\big\{f_t^{\ast }(y^{\ast }):x^{\ast }-y^{\ast }\in (\aff(\dom %
			I_{f}\cap L))^{\circ }\big\}
		\end{align}
	\end{enumerate}
\end{remark}

\begin{corollary}
	\label{corollaryexactsum}In the setting of  Theorem \ref{teorema2}, suppose that 
	one of conditions \ref{QC0} to \ref{QC4.0} holds in a measurable
	set $T_{0}\subset T.$ Then for all $x\in X$\ 
	\begin{equation}\label{eq1:corollaryexactsum}
		\begin{split}
			\partial I_{f}(x)=& \bigcap\limits_{L\in \mathcal{F}(x)}\biggl\{%
			\int_{T_{0}}\left( \partial f_{t}(x)+(\aff{(\dom I_f \cap L-x )}%
			)^{\bot }\right) d\mu (t)\\
			&\hspace{2cm}+\int_{T_{0}^{c}}\partial (f_{t}+\delta _{%
				\aff\{L\cap \dom I_{f}\}})(x)d\mu (t) 
			+N_{L\cap \dom I_{f}}(x)\biggr\},
		\end{split}%
	\end{equation}%
	where $T_{0}^{c}$ is the complement of $T_{0}$. In particular, if $T$ is finite, then 
	\begin{equation}\label{eq2:corollaryexactsum}
		\partial \biggl(\sum\limits_{t\in T}f_{t}\biggr)(x)=\bigcap\limits_{\epsilon
			>0}\cl^{w^*} \biggl\{\sum_{t\in T_{0}}\partial f_{t}(x)+\sum_{t\in
			T_{0}^{c}}\partial _{\epsilon }f_{t}(x)\biggr\}.	\end{equation}
\end{corollary}

\begin{proof}
	Equation \eqref{eq1:corollaryexactsum} is direct from  Theorem \ref{teorema2} and  Remark \ref{remcon}, and so, we only need to prove  \eqref{eq2:corollaryexactsum}. Fix $x\in X,$ $V\in \mathcal{N}_{0}$ and choose $%
	L\in \mathcal{F}(x)$ such that $L^{\perp }\subseteq V$. We may assume that $%
	\partial I_{f}(x)\neq \emptyset .$ By \eqref{eq1:corollaryexactsum}, and
	taking into account \cite[Theorem 3.1]{MR1330645}, we have, for every $%
	\epsilon >0,$ 
	\begin{align*}
		\partial \biggl(\sum\limits_{i\in T}f_{i}\biggr)(x) \subseteq& \sum_{i\in
			T_{0}}\partial f_{i}(x)+(\aff{(\dom I_f \cap L-x )})^{\bot }\\&+\sum_{i\in
			T_{0}^{c}}\partial (f_{t}+\delta _{\aff\{L\cap \dom %
			I_{f}\}})(x)+N_{L\cap \dom I_{f}}(x) \\
		\subseteq& \sum_{i\in T_{0}}\partial f_{i}(x)+(\aff{(\dom I_f \cap L-x )}%
		)^{\bot }+\sum_{i\in T_{0}^{c}}\partial _{\epsilon }f_{t}(x)\\& +(\aff{(\dom I_f
			\cap L-x )})^{\bot }+N_{L\cap \dom I_{f}}(x)+V;
	\end{align*}%
	hence, $\partial f_{i}(x)\neq \emptyset $ for all $i\in T_{0}.$ But we have  $$(%
	\aff{(\dom I_f \cap L-x )})^{\bot }+(\aff{(\dom I_f \cap L-x )})^{\bot
	}+N_{L\cap \dom I_{f}}(x)\subseteq N_{L\cap \dom I_{f}}(x),$$
	and so  
	\begin{equation*}
		\partial \biggl(\sum\limits_{i\in T}f_{i}\biggr)(x)\subseteq \sum_{i\in
			T_{0}}\partial f_{i}(x)+\sum_{i\in T_{0}^{c}}\partial _{\epsilon
		}f_{t}(x)+N_{L\cap \dom I_{f}}(x)+V.
	\end{equation*}%
	Moreover, since (see \cite[Lemma 11]{MR2448918})
	\begin{equation*}
		N_{L\cap \dom I_{f}}(x)=\biggl[ \cl \biggl(\sum_{i\in
			T_{0}}\partial f_{i}(x)+\sum_{i\in T_{0}^{c}}\partial _{\epsilon
		}f_{t}(x)+L^{\perp }\biggr)\bigg]_{\infty },
	\end{equation*}%
	it follows that 
	\begin{equation*}
		\partial \biggl(\sum\limits_{i\in T}f_{i}\biggr)(x)\subseteq \sum_{i\in
			T_{0}}\partial f_{i}(x)+\sum_{i\in T_{0}^{c}}\partial _{\epsilon
		}f_{t}(x)+V+V,
	\end{equation*}%
	which in turn implies that 
	\begin{align*}
		\partial \biggl(\sum\limits_{i\in T}f_{i}\biggr)(x)& \subseteq
		\bigcap\limits_{\epsilon >0,\,i\in T_{0}^{c}}\bigcap\limits_{V\in \mathcal{N}%
			_{0}}\biggl[\sum_{i\in T_{0}}\partial f_{i}(x)+\sum_{i\in T_{0}^{c}}\partial
		_{\epsilon }f_{t}(x)+V+V\biggr] \\
		& =\bigcap\limits_{\epsilon >0,\,i\in T_{0}^{c}} \cl \biggl\{\sum_{i\in
			T_{0}}\partial f_{i}(x)+\sum_{i\in T_{0}^{c}}\partial _{\epsilon }f_{t}(x)%
		\biggr\}.
	\end{align*}%
	This yields the direct inclusion ``$\subset $" and then completes the proof,
	since\ the opposite inclusion is easily  checked.\qed
\end{proof}

\begin{corollary}\label{corollarylfinitedim}
	Let $f:T \times \mathbb{R}^n \to \Rex$ be a convex normal integrand. Assume that $\ri(\dom f_{t})\cap \aff ( \dom
	I_f) \neq \emptyset $ for almost all $t\in T$. Then for every $x\in \R^n$
	
	\begin{align}\label{eqcorollarylfinitedim}
		\partial I_{f}(x)= \int_{T}\left( \partial f_{t}(x)+ N_{{\dom}I_{f}}(x)\right) d\mu (t) .
	\end{align}
\end{corollary}
\begin{proof}
	Fix $x\in \R^n$. Using Corollary \ref{corollaryexactsum}, we have that 
	\begin{align*}
		\partial I_{f}(x)&= \int_{T}\left( \partial f_{t}(x)+(\aff{(\dom I_f -x )}%
		)^{\bot }\right) d\mu (t)  + N_{{\dom}I_{f}}(x)\\
		& =\int_{T}\left( \partial f_{t}(x)+(\aff{(\dom I_f -x )}%
		)^{\bot }+ N_{{\dom}I_{f}}(x)\right) d\mu (t) \\
		&=\int_{T}\left( \partial f_{t}(x)+ N_{{\dom}I_{f}}(x)\right) d\mu (t) .
	\end{align*}\qed
\end{proof}
\begin{remark} \label{remark:corollarylfinitedim}
	Example \ref{ejemploimportante} shows that (\ref{eqcorollarylfinitedim}) cannot be simplified to \begin{align*}
		\partial I_{f}(x)= \int_{T} \partial f_{t}(x) d\mu (t) +N_{{\dom}I_{f}}(x).
	\end{align*}\end{remark} 
\begin{example}
	The main feature of the finite sum given in  \eqref{eq2:corollaryexactsum} is that the characterization of $\partial I_{f}(x)$
	does not involve the normal cone $N_{ \dom I_{f}\cap L}(x).$ This
	fact is specific\ to this finite case and cannot be true in general, even
	for\ smooth data functions $f_{t}$ with\ $\int_T \partial f_t(x)d\mu (t)\neq
	\emptyset $. For example, consider the  Lebesgue measure  on $]0,1]$ and the
	integrand $f:\,]0,1]\times \mathbb{R}\rightarrow \mathbb{R}$ given by $$%
	f(t,x)=x^{2}/t.$$ Then we obtain  $I_{f}=\delta _{\{0\}}$ and, so, $%
	\partial f_t(0)=\{0\}$, while\ $\partial I_{f}(0)=\mathbb{R}$. The same
	example can be adapted to construct a counterexample  for a countable measure over the measurable space 
	$(\mathbb{N},\mathcal{P}(\mathbb{N}))$.
\end{example}

Next, we give  another formula for the subdifferential of finite sums of
convex functions, where a qualification condition involving the relative
interiors is satisfied by only a part of the family $\{f_{t},$ $t\in T\}.$
We need the following technical Lemma, which is an adaptation of classical
techniques in the finite-dimension setting (see, e.g., \cite[Corollary 2.3.5]{MR1921556}).

\begin{lemma}
	\label{lemmacontinuity}Let $g\in \Gamma _{0}(X)$ and let $L\subset X$ be a
	finite-dimensional affine subspace. If $g$ is continuous relative to $\aff(%
	\dom g)$ at some point in $ \dom g\cap L$, then for every $%
	x^{\ast }\in X^{\ast }$ 
	\begin{equation*}
		\big(g+\delta _{L}\big)^{\ast }(x^{\ast })=\min \big\{g^{\ast }(y^{\ast
		})+\delta _{L}^{\ast }(x^{\ast }-y^{\ast }):y^{\ast }\in X^{\ast }\big\}.
	\end{equation*}
\end{lemma}

\begin{corollary}
	\label{qualfin}Assume that $T$ is finite and let $\{f_{t}\}_{t\in
		T}\subseteq \Gamma _{0}(X)$ and $T_{0}\subset T$ be given. Then the
	following statements hold true\emph{:}
	
	\emph{(i) }If $\bigcap\limits_{t\in T_{0}}\ri_{\aff( \dom f_{t})}(%
	\dom f_{t})\cap \bigcap\limits_{t\in T_{0}^{c}} \dom %
	f_{t}\neq \emptyset $ and each $f_{t}$ with  $t\in T_{0}$ is continuous on $\ri_{%
		\aff( \dom f_{t})}( \dom f_{t})$, then for every $x\in X$ 
	\begin{equation*}
		\partial \biggl(\sum\limits_{t\in T}f_{t}\biggr)(x)=\bigcap\limits_{\epsilon
			>0} \cl^{w^*} \biggl\{\sum\limits_{t\in T_{0}}\partial
		f_{t}(x)+\sum\limits_{t\in T_{0}^{c}}\partial _{\epsilon }f_{t}(x)\biggr\}.
	\end{equation*}
	
	\emph{(ii) }If $\bigcap\limits_{t\in T_{0}}\inte (%
	\dom f_{t})\cap \bigcap\limits_{t\in T_{0}^{c}} \dom %
	f_{t}\neq \emptyset $ and each $f_{t}$ with $t\in T_{0}$ is continuous on $\inte%
	( \dom f_{t})$, then\ for every $x\in X$ 
	\begin{equation*}
		\partial \biggl(\sum\limits_{t\in T}f_{t}\biggr)(x)=\sum\limits_{t\in T_{0}}\partial
		f_{t}(x)+\bigcap\limits_{\epsilon
			>0} \cl^{w^*} \biggl\{\sum\limits_{t\in T_{0}^{c}}\partial _{\epsilon }f_{t}(x)\biggr\}.
	\end{equation*}
\end{corollary}

\begin{proof}
	Let $f:=\sum_{t\in T} f_t$. Fix $t \in T_0$ and consider $L \in \mathcal{F}(x)$. Then, by applying  Lemma \ref{lemmacontinuity} with $g=f_t$  and  $L=\aff(\dom f \cap L)$,  we ensure the validity of condition \eqref{QC4}. Therefore
	statement (i) follows by\ applying\  Corollary \ref{corollaryexactsum}. Statement (ii) follows from (i) by arguing as in	Lemma \ref{lemmacontinuity}.\qed
\end{proof}

\section{Suslin spaces or discrete measure space\label{SUSLINSPACES}}

In this section, we give more sharp characterizations of the $\varepsilon $%
-subdifferential of $I_{f}$ under the cases where either $X,$ $X^{\ast }$
are Suslin spaces, or $(T,\Sigma )=(\mathbb{N},\mathcal{P}(\mathbb{N}))$.
These settings, indeed, permit the use of measurable selection theorems, which
give us more control over the integration of the multifunctions $\partial
_{\varepsilon (t)}f_{t}(x)$ and $N_{\dom I_{f}\cap L}^{\epsilon
}(x). $ We recall that $f:T\times X\rightarrow \mathbb{R\cup \{+\infty \}}$
is a given normal convex integrand, and $(T,\Sigma ,\mu )$ is a nonnegative complete $%
\sigma $-finite measure space. The function $I_{f}:X\rightarrow \mathbb{%
	R\cup \{+\infty \}}$ is defined as 
\begin{equation*}
	I_{f}(x)=\int_{T}f(t,x)d\mu (t).
\end{equation*}%
The following corollary rewrites  the characterization given in\
Theorem \ref{teorema2} by using only the $\varepsilon $-subdifferential of
the $f_{t}$'s.

\begin{theorem}
	\label{corollary1}We suppose\ that either $X$ and $X^{\ast }$ are Suslin
	spaces or $(T,\Sigma )=(\mathbb{N},\mathcal{P}(\mathbb{N}))$. Then for every 
	$x\in X$ and $\varepsilon \geq 0$ we have%
	\begin{equation*}
		\partial _{\epsilon }I_{f}(x)=\bigcap\limits_{L\in \mathcal{F}%
			(x)}\hspace*{-0.1cm} \bigcup\limits_{\substack{ \epsilon _{1},\epsilon _{2}\geq 0  \\ %
				\epsilon =\epsilon _{1}+\epsilon _{2}  \\ \ell \in \mathcal{I}(\epsilon
				_{1})}}\bigcap\limits_{\eta \in
			L^{1}(T,\mathbb{R}^*_+)} \hspace*{-0.2cm} \cl^{\beta (X^{\ast },X)}  \left\{ \int_{T}\left( \partial
		_{\ell (t)+\eta (t)}f_{t}(x)+N_{ \dom  I_{f}\cap L}^{\epsilon
			_{2}}(x)\right) d\mu (t)\right\}.
	\end{equation*}%
\end{theorem}

\begin{proof}
	We only need to prove the inclusion  ``$\subseteq$"  in which we suppose that $%
	x=0.$\ Take $F\in \mathcal{F}(0),$ $\epsilon >0$ and $$L:=\spn\{F\cap 
	\dom  I_{f}\}\,\,(=\aff\{F\cap  \dom  I_{f}\}),$$ say\ $L=\spn%
	\{e_{i}\}_{1}^{p}$, with\ $\{e_{i}\}_{1}^{p}$ being\ linearly independent, and $\co\{\pm e_{i}\}_{i=1}^{p}$ is the unit closed ball in $L$ with
	respect to a norm $\Vert \cdot \Vert _{L}$ (on $L).$ Let\ $P:X\rightarrow L$
	be a continuous projection and let $W\in \mathcal{N}_{0}$ be as in
	Lemma \ref{lemma2}. Given $\delta >0,$ we pick an integrable function $\gamma
	:T\rightarrow (0,+\infty )$ such that 
	\begin{equation}
		\int_{T}\gamma (t)d\mu \leq \delta ,  \label{p}
	\end{equation}%
	and\ define the measurable multifunctions $U,$ $V:T\rightrightarrows L^{\ast
	}$ as%
	\begin{align*}
		U(t):=&\{x^{\ast }\in X^{\ast }:|\langle x^{\ast },e_{i}\rangle |\leq \gamma
		(t),\;i=1,...,p\},\\V(t):=&\{x^{\ast }\in L^{\ast }:|\langle x^{\ast
		},e_{i}\rangle |\leq \gamma (t),\;i=1,...,p\}.
	\end{align*}%
	Now, take\ $x^{\ast }\in \partial _{\varepsilon }I_{f}(0)$ and fix\ a
	positive measurable function $\eta.$ By formula \eqref{formula2.1} in Theorem \ref{teorema2} there exist $\epsilon _{1},\epsilon _{2}\geq 0$ with $%
	\epsilon _{1}+\epsilon _{2}=\epsilon $, $\ell \in \mathcal{I}(\epsilon _{1})$%
	, an integrable selection $x_{L,\epsilon }^{\ast }(\cdot )$ of the
	multifunction $t\rightarrow \partial _{\ell (t)}\left( f_{t}+\delta
	_{L}\right) (0),$ and $\lambda ^{\ast }\in N_{ \dom  I_{f}\cap
		L}^{\epsilon _{2}}(0)$ such that $$x^{\ast }=\int_{T}x_{L,\epsilon }^{\ast
	}(t)d\mu (t)+\lambda ^{\ast }.$$ Since $U(t)\in \mathcal{N}_{0}$, by \cite[%
	Theorems 3.1 and 3.2]{MR1330645} we have that, for ae $t\in T,$ 
	\begin{eqnarray*}
		x_{L,\epsilon }^{\ast }(t)\in \partial _{\ell (t)}\left( f_{t}+\delta
		_{L}\right) (0) &\subset &\partial _{\ell (t)+\eta (t)}f_{t}(0)+L^{\perp
		}+U(t) \\
		&\subset &\partial _{\ell (t)+\eta (t)}f_{t}(0)+N_{ \dom  I_{f}\cap
			L}(0)+P^{\ast }(V(t)).
	\end{eqnarray*}%
	We define the multifunction $G:T\rightrightarrows X^{\ast }\times X^{\ast
	}\times L^{\ast }$ as\ 
	\begin{equation*}
		(y^{\ast },w^{\ast },v^{\ast })\in G(t)\Leftrightarrow \left\{ 
		\begin{array}{c}
			y^{\ast }\in \partial _{\ell (t)+\eta (t)}f(t,0),\;w^{\ast }\in N_{{%
					\dom}I_{f}\cap L}(0),\;\text{and }v^{\ast }\in V(t), \\ 
			x_{L,\epsilon }^{\ast }(t)=y^{\ast }+w^{\ast }+P^{\ast }(v^{\ast }).%
		\end{array}%
		\right.
	\end{equation*}%
	If $X,X^{\ast }$ are Suslin spaces, then, by  Lemma \ref{medibilidadgrafo}, $G$
	is measurable, and so, by\  Proposition \ref{measurableselectiontheorem} it
	admits a measurable selection $(y^{\ast }(\cdot ),w^{\ast }(\cdot ),v^{\ast
	}(\cdot )).$ This also obviously holds when $(T,\Sigma )=(\mathbb{N},%
	\mathcal{P}(\mathbb{N})).$ Thus, by  Lemma \ref{lemma2} the\ function\ $%
	u^{\ast }(t):=v^{\ast }(t)\circ P$ is integrable and we get\ 
	\begin{equation*}
		\sigma _{W}(u^{\ast }(t))\leq \max\limits_{i=1,...,p}\langle v^{\ast
		}(t),e_{i}\rangle \leq \gamma (t)\text{ for ae }t\in T.
	\end{equation*}%
	Consequently, the function $y^{\ast }+w^{\ast }=x_{L,\epsilon }^{\ast
	}(\cdot )-u^{\ast }(\cdot )$ is strongly integrable and we have (recall \eqref{p}) 
	\begin{eqnarray*}
		\sigma _{W}(x^{\ast }-\int_{T}(y^{\ast }(t)+w^{\ast }(t))d\mu
		(t)-\lambda ^{\ast }) &=&\sigma _{W}(\int_{T}x_{L,\epsilon }^{\ast }(t)d\mu
		(t)-\int_{T}(y^{\ast }(t)+w^{\ast }(t))d\mu (t)) \\
		&=&\sigma _{W}(\int_{T}u^{\ast }(t)d\mu (t)) \\
		&\leq &\int_{T}\sigma _{W}(u^{\ast }(t))d\mu (t) \\
		&\leq &\int_{T}\gamma (t)d\mu \leq \delta ;
	\end{eqnarray*}%
	that is,%
	\begin{equation*}
		x^{\ast }-\int_{T}(y^{\ast }(t)+w^{\ast }(t))d\mu (t)-\lambda ^{\ast
		}\in \delta W^{\circ },
	\end{equation*}%
	and, due to the arbitrariness of $\delta $, 
	\begin{equation*}
		x^{\ast }\in  \cl^{\beta (X^{\ast
			},X)}\left(\int_{T}\left( \partial_{\ell (t)+\eta (t)}f_{t}(0)+N_{	\dom}I_{f}\cap L (0)\right) d\mu (t)+N_{ \dom  I_{f}\cap L}^{\epsilon_{2}}(0)\right).
	\end{equation*}%
	Finally, since $N_{ \dom  I_{f}\cap L}^{\epsilon _{2}}(0)\subset
	\int_{T}N_{ \dom  I_{f}\cap L}^{\epsilon _{2}}(0)d\mu (t)$ we
	conclude that 
	\begin{equation*}
		x^{\ast }\in  \cl  \nolimits^{\beta (X^{\ast
			},X)}\int_{T}\left( \partial _{\ell (t)+\eta (t)}f_{t}(0)+N_{{
				\dom}I_{f}\cap L}^{\epsilon _{2}}(0)\right) d\mu (t).
	\end{equation*}\qed
\end{proof}

The next result is a finite-dimensional-like characterization of the
subdifferential of $I_{f}.$ Recall that a closed affine subspace\ $A\subset
X $ is said to have a continuous projection if there exists an affine continuous
projection from $X$ to $A$, or equivalently, if there exists a continuous linear projection from $X$ to $A-x_0$, where $x_0\in A$.

\begin{theorem}\label{theoremfinitedimension}
	Let\ $X$, $X^{\ast }$ and $T$ be as in  Theorem \ref{corollary1}. If $%
	I_{f}$ is continuous on $\ri({\dom}I_{f})\neq \emptyset $ and $%
	\overline{\aff}(${${\dom}$}${I_{f}})$ has a continuous projection,
	then 
	\begin{equation}
		\partial I_{f}(x)=\bigcap\limits_{\eta \in L^{1}(T,\mathbb{R}^*_+)}{\cl}^{w^{\ast }}\left\{ (w)\text{-}\int_{T}\left( \partial
		_{\eta (t)}f_{t}(x)+N_{{\dom}I_{f}}(x)\right) d\mu (t)\right\} .\label{newone}
	\end{equation}
	
\end{theorem}

\begin{proof}
	Because the inclusion  ``$\subseteq$"  is immediate\ we only need to prove the
	other inclusion ``$\subseteq$" when\ $x=0;$ hence,\ $F:=\overline{\aff}(${$%
		{\dom}$}${I_{f}})$ is a closed subspace of $X.$ Let $x^{\ast }\in
	\partial I_{f}(0)$, $\eta \in L^{1}(T,(0,+\infty )),$ and $V:=\{h^{\ast }\in
	X^{\ast }:|\langle h^{\ast },e_{i}\rangle |\leq 1,\;\;i=1,...,p\}$ for some $%
	\{e_{i}\}_{i=1}^{p}\subset X$. By the current assumption, we take $x_{0}\in %
	\ri({\dom}I_{f})$ and a continuous projection $P:X\rightarrow F$.
	Define $L=\spn\{e_{i},P(e_{i}),x_{0}\}_{i=1}^{p}$ and $$W(t):=\{h^{\ast }\in
	X^{\ast }:\max \{|\langle h^{\ast },e_{i}\rangle |,|\langle h^{\ast
	},P(e_{i})\rangle |,|\langle h^{\ast },x_{0}\rangle |\}\leq \epsilon (t),
	i=1,...,p\},$$ where $\epsilon (\cdot )$ is any positive integrable function
	with values on $(0,1)$ and $\int_{T}\epsilon d\mu \leq 1/2$. Then $$L^{\perp
	}+W(t)\subseteq W(t)\subseteq V.$$ Because $L\cap \ri({\dom}I_{f})\neq
	\emptyset $, we have (see, e.g.,  Corollary \ref{qualfin}) 
	\begin{equation}
		N_{{\dom}I_{f}\cap L}(0)={\cl}^{w^{\ast }}\big(L^{\perp }+N_{%
			{\dom}I_{f}}(0)\big).  \label{equation:doce}
	\end{equation}%
	By  Theorem \ref{corollary1} there exists a (strong) integrable selection $%
	y^{\ast }(t)\in \partial _{\eta (t)}f_{t}(0)+N_{{\dom}I_{f}\cap
		L}(0)\subset \partial _{\eta (t)}f_{t}(0)+N_{{\dom}I_{f}}(0)+W(t),$
	due to \eqref{equation:doce}, such that 
	\begin{equation}
		x^{\ast }-\int_{T}y^{\ast }d\mu \in V.  \label{ss}
	\end{equation}%
	Also, by the measurability of multifunctions $\partial _{\eta (\cdot
		)}f_{\cdot }(0),$ $N_{{\dom}I_{f}}(0),$ and $W(\cdot )$ (see, e.g., 
	\cite{MR0372610}), there exists a (weakly) measurable selection $z^{\ast
	}(\cdot )$ of $\partial _{\eta (\cdot )}f_{\cdot }(0)+N_{{\dom}%
		I_{f}}(0)$ such that 
		$$y^{\ast }(t)-z^{\ast }(t)\in W(t)\, \text{ for ae }
		$$
		 (the
	existence of such a selection is guaranteed for Suslin spaces by the
	representation theorem of Castaing, while it is straightforward in the
	discrete case). 
	
	Let us verify that the function $z^{\ast }(\cdot )\circ P$
	is weakly integrable: Given $U\in \mathcal{N}_{0}$ such that $%
	x_{0}+P(U)\subset \ri({\dom}I_{f})$ (using the the continuity of $P$)
	we have, for every $y\in U,$ 
	\begin{align}
		\langle z^{\ast }(t)\circ P,y\rangle & =\langle z^{\ast }(t),Py\rangle 
		\notag \\
		& =\langle z^{\ast }(t),x_{0}+Py\rangle -\langle z^{\ast }(t),x_{0}\rangle 
		\notag \\
		& \leq f(t,x_{0}+P(y))-f(t,0)+\eta (t)-\langle z^{\ast }(t),x_{0}\rangle
		+\sigma _{N_{{\dom}I_{f}}(0)}(x_{0}+P(y))  \notag \\
		& \leq f(t,x_{0}+P(y))-f(t,0)+\eta (t)+|\langle z^{\ast }(t)-y^{\ast
		}(t),x_{0}\rangle |+|\langle y^{\ast }(t),x_{0}\rangle |  \notag \\
		& \leq f(t,x_{0}+P(y))-f(t,0)+\eta (t)+\epsilon (t)+|\langle y^{\ast
		}(t),x_{0}\rangle |,  \label{in}
	\end{align}%
	and the weak integrability of $z^{\ast }(\cdot )\circ P$\ follows, as 
	\begin{align*}
		\int_{T}\left\vert \langle z^{\ast }(t)\circ P,y\rangle \right\vert d\mu
		(t)\leq& \int_{T}\left\vert f(t,x_{0}+P(y))\right\vert d\mu
		(t)\\ &+\int_{T}\left\vert f(t,x_{0}-P(y))\right\vert d\mu
		(t)-I_{f}(0)\\&+\int_{T}(\eta (t)+\epsilon (t)+|\langle y^{\ast
		}(t),x_{0}\rangle |)d\mu (t)<+\infty .
	\end{align*}%
	Moreover, \eqref{in} implies that\ $\int_{T}z^{\ast }\circ Pd\mu $ is
	uniformly bounded on a neighborhood of zero, so that $$\int_{T}z^{\ast
	}\circ Pd\mu \in X^{\ast }.$$ Finally, we have $z^{\ast }(t)\circ P=z^{\ast }(t)+z^{\ast }(t)\circ P-z^{\ast }(t)$,  
	\begin{align*}
		z^{\ast }(t) \in
		\partial _{\eta (t)}f_{t}(0)+N_{{\dom}I_{f}}(0) \text{ and }	z^{\ast }(t)\circ P-z^{\ast }(t) \in F^\perp,
	\end{align*}%
	and, consequently,
	\begin{align*}
		z^{\ast }(t)\circ P \in \partial _{\eta (t)}f_{t}(0)+N_{{\dom}I_{f}}(0)	+F^\perp =\partial _{\eta (t)}f_{t}(0)+N_{{\dom}I_{f}}(0)
	\end{align*}
	and (recall \eqref{ss}) 
	\begin{equation*}
		x^{\ast }-\int_{T}z^{\ast }\circ Pd\mu =x^{\ast }-\int_{T}y^{\ast }d\mu
		+\int_{T}(y^{\ast }-z^{\ast })d\mu -\int_{T}(z^{\ast }\circ P-z^{\ast })d\mu
		\in V+V+F^{\perp }.
	\end{equation*}%
	So, observing that $F^{\perp }\subset \int_{T}N_{{\dom}%
		I_{f}}(0)d\mu (t)$,
	\begin{eqnarray*}
		x^{\ast } &\in &(w)-\int_{T}\left( \partial _{\eta (t)}f_{t}(0)+N_{%
			{\dom}I_{f}}(0)\right) d\mu (t)+V+V+F^{\perp } \\
		&\subset &(w)-\int_{T}\left( \partial _{\eta (t)}f_{t}(0)+N_{{
				\dom}I_{f}}(0)\right) d\mu (t)+V+V.
	\end{eqnarray*}%
	Hence, by intersecting over $V$ we get 
	\begin{equation*}
		x^{\ast }\in {\cl}^{w^{\ast }}\left\{ (w)\text{-}\int_{T}%
		\left( \partial _{\eta (t)}f_{t}(x)+N_{{\dom}I_{f}}(x)\right) d\mu
		(t)\right\} ,
	\end{equation*}%
	which gives the desired inclusion due to the arbitrariness of the function $\eta
	.$\qed
\end{proof}
\begin{remark}
	If in Theorem \ref{theoremfinitedimension}, $X$ is a Banach space and, for a given point $\bar{x}$, there exists some sequence $(\eta_n)_n \subset L^1(T,(0,+\infty))$ such that 
	$$
	\int_T \eta (t ) d\mu(t) \to 0,\quad (w)\text{-}\int_{T} \partial
	_{\eta_n (t)}f_{t}(\bar{x})  d\mu (t)\neq \emptyset, \, \forall n\in \N,
	$$ 
	then 
	\begin{align}\label{newone2}
		\partial I_{f}(\bar{x})=\bigcap\limits_{n \in \N }{\cl}^{w^{\ast }}\left\{ (w)\text{-}\int_{T}  \partial
		_{\eta_n (t)}f_{t}(\bar{x})d\mu (t) +N_{{\dom}I_{f}}(\bar{x}) \right\} .
	\end{align}
	Indeed, the inclusion ``$\subseteq$" is straightforward.  Moreover, for each $n \in \N$ we have that  (due to the fact that the integral of $  \partial
	_{\eta_n (\cdot)}f_{\cdot}(\bar{x}) $ and $N_{{\dom}I_{f}}(\bar{x})$ are nonempty) 
	\begin{align*}
		(w)\text{-}\int_{T}\left( \partial
		_{\eta (t)}f_{t}(x)+N_{{\dom}I_{f}}(x)\right) d\mu (t) 
		& \subseteq	{\cl}^{w^{\ast }}\left\{	(w)\text{-}\hspace{-0.1cm}\int_{T}  \partial
		_{\eta_n (t)}f_{t}(\bar{x})d\mu (t) +N_{{\dom}I_{f}}(\bar{x}) \right\}
	\end{align*}
	(see e.g. \cite[Proposition 5.6]{MR1485775}, \cite[Proposition 8.6.2]{MR2458436}). Then, by \eqref{newone}, the left-hand side is included in the right-hand side.
\end{remark}
 
 \begin{corollary}\label{corollaryfinitedimension2}
	Let $(T,\Sigma,\mu)$ be a finite-measure space and 	let $f:T \times \mathbb{R}^n \to \Rex$ be a convex normal integrand. Then, for every $x\in \R^n$,
	
	\begin{align}\label{eqfinal}
		\partial I_{f}(x)=\bigcap\limits_{\eta >0 }\cl\left\{\int_{T}\left( \partial
		_{\eta }f_{t}(x)+N_{{\dom}I_{f}}(x)\right) d\mu (t)\right\} .
	\end{align}
	
\end{corollary}
\begin{proof}
	Since the measure is finite we have that the right-hand side of \eqref{eqfinal} is included in $\partial I_{f}(x)$. Moreover, since the constant positive functions belongs to $L^1(T,(0,+\infty))$, Theorem \ref{theoremfinitedimension} implies that the subdifferential of $I_f$ at $x$  is included in  the  right-hand side of \eqref{eqfinal}.
\qed	
\end{proof}

\section{Conclusions}
\label{sec:conclusions}
In this work, we presented general characterizations of the $\varepsilon$-subdifferential of the integral functional $I_f$, without any qualification conditions, when defined over a locally convex space (see Theorem \ref{teorema2} for the main result). We used a finite-dimensional reduction approach since there is no theory about measurable selections and integration of multifunctions in non-separable locally convex space. We provided simplifications under qualification conditions of the nominal data (see Corollary \ref{corollaryexactsum}). 

\,

\,

{{\bf{Acknowledgements.}} We very much appreciate the insightful suggestions and helpful comments of the referees, which have contributed to improving the current revision of the manuscript.}

\bibliographystyle{plain}
\bibliography{references}
\end{document}